\documentclass[english, 12pt]{amsart}
\usepackage{babel}
\usepackage[]{geometry}

\usepackage{caption}
\usepackage{subcaption}
\usepackage{amsmath,amssymb,amsfonts,latexsym}
\usepackage{graphicx}
\usepackage{xcolor}
\usepackage{etoolbox}
\usepackage{placeins}
\usepackage{needspace}
\usepackage{tikz}
\usepackage{booktabs}

\newcommand{\R}{\mathbb{R}}
\newcommand{\N}{\mathbb{N}}

\newcommand{\jump}[1]{[\![ #1 ]\!]}

\def\1{\raisebox{2pt}{\rm{$\chi$}}}

\theoremstyle{plain}

\newtheorem{proposition}{Proposition}

\newtheorem{theorem}{Theorem}
\newtheorem{corollary}{Corollary}

\theoremstyle{definition}
\newtheorem{example}{Example}

\theoremstyle{remark}

\author{C.~Esteve-Yag\"ue}
\address{\textbf{Carlos Esteve-Yag\"ue} \newline
Departamento de Matem\'aticas, Universidad de Alicante, Spain}
\email{c.esteve@ua.es}

\author{R.~Tsai}
\address{\textbf{Richard Tsai} \newline
Department of Mathematics and Oden Institute, The University of Texas at Austin, USA}
\email{ytsai@math.utexas.edu}

\author{S.~Osher}
\address{\textbf{Stanley Osher} \newline
Department of Mathematics, UCLA, USA.}
\email{sjo@math.ucla.edu}


\usepackage{hyperref}

\newcommand{\nb}[3]{
{\colorbox{#2}{\bfseries\sffamily\scriptsize\textcolor{white}{#1}}}
{\textcolor{#2}{\sf\small$\blacktriangleright$\textit{#3}$\blacktriangleleft$}}}
\newcommand{\carlos} [1]{\nb{Carlos}{white!30!purple}{#1}}
\newcommand{\richard}[1]{\nb{Richard}{red!30!green}{#1}}

\newif\ifmanuscriptreview
\ifdefined\ManuscriptClean
  \manuscriptreviewfalse
\else
  \manuscriptreviewtrue
\fi
\renewcommand{\carlos}[1]{}
\renewcommand{\richard}[1]{}

\begin{document}

\title[Least-squares residual minimization for differential equations]{When does least-squares residual minimization solve a differential equation?}

\date{\today}

\begin{abstract}
Least-squares residual minimization approximates a differential equation through an optimization problem. This paper investigates the conditions required to transition from stationary points and minimizing sequences of the residual functional to the exact solution. Starting from the classical orthogonality relation between the residual and the range of the linearized operator, we study how boundary conditions and parameterized variations 
determine what can be inferred about the residual at a stationary point. Smooth spurious critical points can occur even for well-posed problems; for linear equations, a compatibility condition yields an exact projection characterization. {For} a scalar conservation law with transonic boundary data,
every minimizing sequence converges to a continuous function that
does not solve the equation, even when a stationary entropy shock
exists. We give conditions under which approaching this limit
forces the parameter norms to diverge. A Hamilton--Jacobi example
shows that a function can have zero residual without being the
viscosity solution. Favorable results for a viscous model and uniformly convex Monge--Amp\`ere critical points identify conditions that restore the corresponding implications. {For piecewise-smooth trial families, the regularity needed to define the residual functional does not by itself justify the sampled gradient used in training. Moving residual jumps can contribute terms absent from that gradient; an explicit example has identically vanishing sampled gradients and a nonzero continuous derivative.} {Together, these results clarify what stationarity and sampled gradients
imply about the residual, and under which conditions minimizing sequences
lead to the intended solution.}
\end{abstract}

\maketitle


\section{Introduction}\label{sec:intro}

Least-squares residual minimization provides a common approach to approximating solutions of differential equations. In physics-informed neural networks (PINNs) and related methods, the approximation is sought in a parameterized family of trial functions, such as neural networks. The residual functional is estimated from point samples. A gradient-based optimization method adjusts the parameters to minimize this sampled functional~\cite{bochev1998finite,raissi2019physics}.
Interpreting the resulting approximation requires relating properties of this constrained optimization problem to the differential equation and its intended solution.

This paper asks under what conditions stationarity implies zero residual, and under what conditions minimizing sequences converge to the intended solution of the differential equation. {The stationarity question arises already for the exact residual functional with all admissible variations available}: well-posedness of the differential equation does not by itself ensure that every critical point has zero residual. Parameterization and sampling introduce further questions about which variations are tested and whether the computed gradient is consistent with the first variation of the continuous loss.

These are distinct implications. An estimate that converts a small residual into a small solution error does not by itself show that a small parameter gradient implies a small residual. Accurate estimates of the residual functional do not guarantee accurate estimates of its gradient. Even when an optimization method produces a minimizing sequence, the corresponding trial functions may fail to converge to a function in the admissible class. We organize the analysis around three questions:

\medskip\par\noindent\rule{\textwidth}{0.4pt}
\noindent \textbf{Q1:} \emph{When does a critical point of the continuous least-squares functional necessarily solve the differential equation?}

\noindent \textbf{Q2:} \emph{At a critical point, how does the residual look if it is nonzero?}

\noindent \textbf{Q3:} \emph{Does driving the residual loss down to its infimum necessarily lead to the intended solution?}
\par\noindent\rule{\textwidth}{0.4pt}\par\medskip
For Q1 and Q2 we examine both the full admissible variation space and the variations generated by a parameterization, and distinguish the continuous first variation from the gradient of the sampled loss.
For Q3, we examine whether minimizing sequences converge within the admissible class and whether their limits give the intended solution.

\subsection*{Two motivating examples}
Two elliptic boundary-value problems separate the continuum stationarity question from the questions introduced by parameterization and sampling.

The question of whether a continuous critical point necessarily solves the differential equation already has a negative answer for a simple, well-posed one-dimensional problem.
 Consider
\begin{equation}\label{eq:neumann-semilinear}
-u''+u^3=1\quad\text{in }(0,1),
\qquad
u'(0)=u'(1)=0,
\end{equation}
and define the least-squares residual loss as
\[
J(u):=\frac12\int_0^1(-u''+u^3-1)^2\,dx,
\qquad
\forall u\in X:=\{u\in H^2(0,1):u'(0)=u'(1)=0\}.
\]
Clearly, the constant function $u^*\equiv1$ is a solution of \eqref{eq:neumann-semilinear} and a global minimum of $J$. If $u_1$ and $u_2$ are
two solutions and $w=u_1-u_2$, multiplication of the difference equation
by $w$ gives
\[
\int_0^1|w'|^2\,dx
+
\int_0^1
(u_1-u_2)^2(u_1^2+u_1u_2+u_2^2)\,dx
=0,
\]
and hence $u_1=u_2$. More generally, the Neumann operator
$u\mapsto -u''+u^3$ is coercive and strictly monotone in the natural
variational setting, so the problem is well posed under perturbations of
the forcing.

Taking the first variation of $J$ at $u$, with an admissible test function $v\in X$, we derive  
\begin{equation*}
    DJ(u)[v] = \int_0^1 (-u'' + u^3 - 1)(-v'' + 3u^2v) \, dx = \bigl\langle -u'' + u^3 - 1, -v'' + 3u^2v \bigr\rangle_{L^2(0,1)}.
\end{equation*}
We see that $\bar u\equiv0$ is a smooth critical point of $J$ with $J(\bar u)=1/2$, \emph{and it does not solve the differential equation.}
Indeed, every admissible variation $v\in X$ satisfies
\[
DJ(0)[v]
=
\int_0^1(-1)(-v'')\,dx
=
v'(1)-v'(0)
=
0.
\]
Thus the exact functional $J(u)$ can be positive at a smooth critical point.

At $u=0$, the derivative of the cubic term vanishes, leaving the linearized operator $v\mapsto-v''$. The Neumann boundary conditions ensure that $-v''$ has zero mean and is therefore orthogonal to the constant residual $-1$.

This orthogonality condition is closely related to the Fredholm alternative. At a critical point, the residual belongs to the null space of the adjoint of the linearized operator, with boundary conditions determined by the admissible variations. 
The corresponding adjoint equation and boundary conditions are understood weakly when the residual lacks sufficient regularity for a classical interpretation.
Studying this null space therefore provides conditions under which stationarity implies zero residual and describes the possible nonzero residuals when those conditions fail.
A parameterization of the trial functions further restricts the available variations. The question is whether stationarity under this restriction still forces the residual to vanish.

The second example concerns the Dirichlet--Poisson problem
\begin{equation}\label{eq:model-Poisson-problem}
-u''=1\quad\text{in }(-a,a),\qquad u(\pm a)=0.
\end{equation}
For $J(u)=\tfrac12\int_{-a}^a(-u''-1)^2\,dx$, every critical point in $X:=H^2(-a,a)\cap H_0^1(-a,a)$ solves the equation. 
Indeed, if $u\in X$ is a critical point of $J$, then
$$
DJ(u)[v] = \left\langle -u'' -1, \, -v'' \right\rangle_{L^2(-a,a)} = 0, \quad \forall v\in X.
$$
Since the Dirichlet--Laplace operator reaches every $L^2$ right-hand side, it holds that $\left\langle -u'' -1, \, g \right\rangle_{L^2(-a,a)} = 0$ for all $g\in L^2(-a,a)$, implying $-u'' = 1$. Nevertheless, training a parameterized model with sampled gradients can reach positive-loss plateaus. Figure~\ref{fig:elu-results} illustrates this behavior for an ELU network; the architecture and training protocol are given in the caption.

\begin{figure}[!htbp]
\centering
\begin{tabular}{cc}
\includegraphics[width=0.42\textwidth]{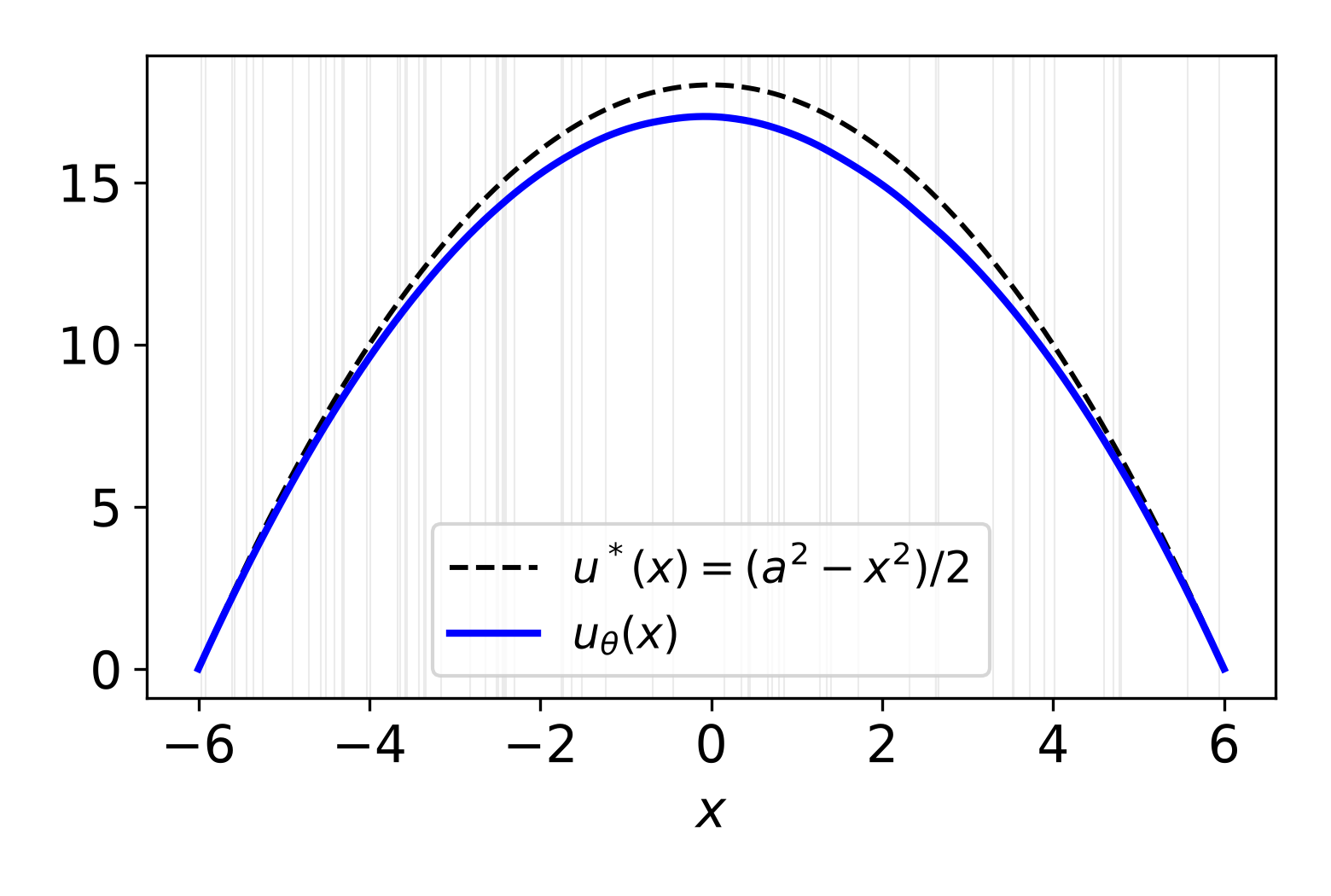}&
\includegraphics[width=0.42\textwidth]{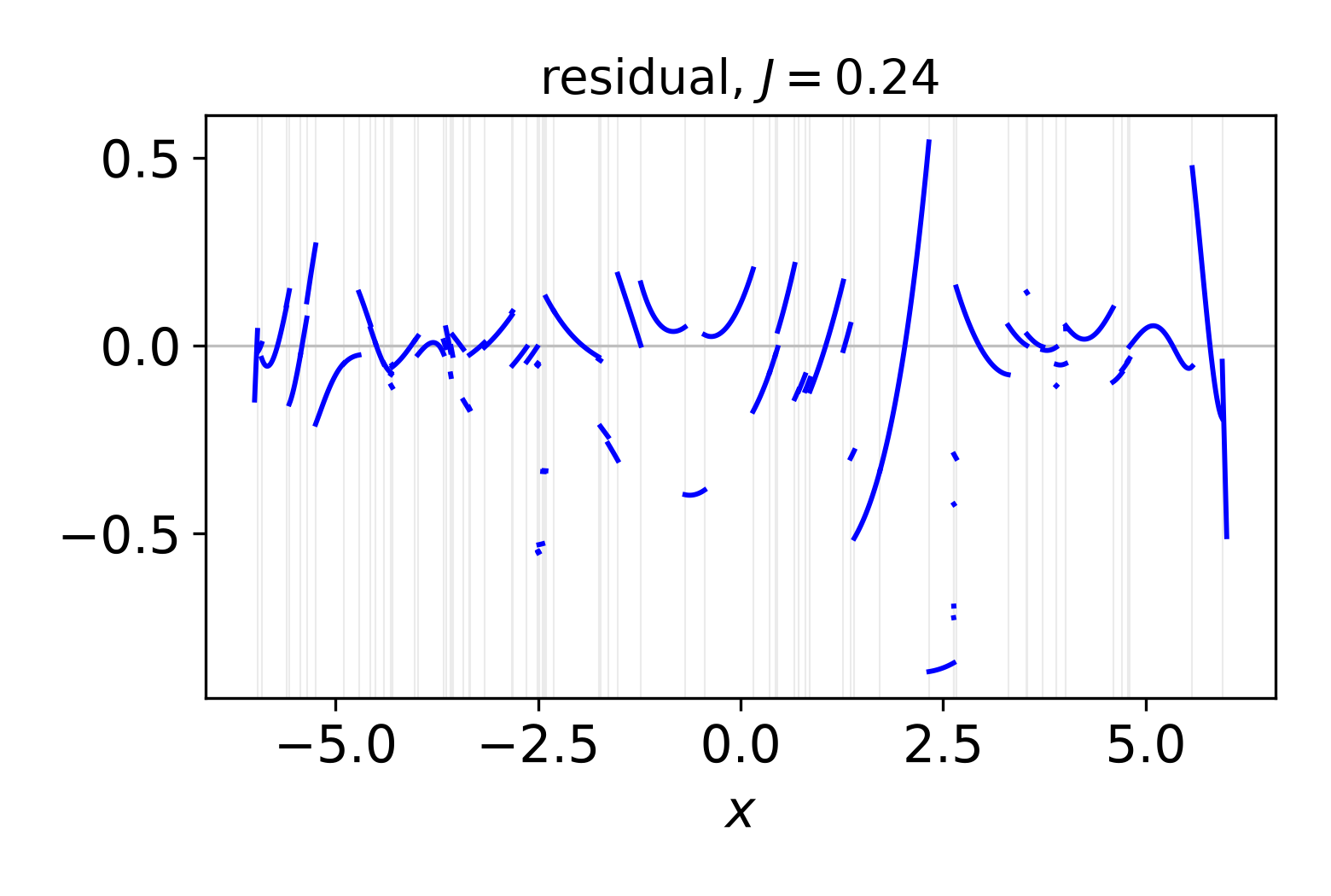}
\end{tabular}
\caption{The Dirichlet Poisson problem with $a=6$. The trial function is $u_\theta(x)=\tanh(x+a)\tanh(a-x)\Psi_\theta(x)$, where $\Psi_\theta$ is an ELU network with two hidden layers of width $50$, so the boundary conditions hold exactly. Training uses Adam on the sampled gradient with ten uniformly resampled points per iteration and learning rate $10^{-3}$ for $50{,}000$ iterations, then $10^{-4}$ and $10^{-5}$ for $25{,}000$ iterations each. Left: the result for seed $3$ after $10^5$ iterations, compared with the exact solution. Right: the residual, with the interfaces of the network, the points where $u_\theta''$ jumps, marked by light vertical lines. The independently evaluated integral loss is approximately $0.241$. }
\label{fig:elu-results}
\end{figure}

The favorable function-space result does not identify the cause of a training plateau. One must also examine the variations generated by the parameterized functions (neural networks, for example) and the relation between the sampled gradient and the continuous first variation. Section~\ref{sec:piecewise-training} derives the interface contribution that can separate these gradients and gives an exact analytical example.

\FloatBarrier

\subsection*{Relation to existing work and contributions}

Normal equations and residual orthogonality are standard tools in classical least-squares methods, see e.g.~\cite{bochev1998finite}. In the PINNs literature, analyses of optimization failures and stiffness~\cite{krishnapriyan2021characterizing,wang2021understanding}, convergence of minimizers under regularity assumptions~\cite{Shin-et-alconvergence}, {error estimates in terms of training and quadrature errors~\cite{mishra2023generalization}}, and {conditions under which stationarity implies zero residual at fixed collocation points}~\cite{hosseiniDashtbayaz2024critical} address several of the implications considered here. {For scalar hyperbolic conservation laws, weak PINNs use residuals based on entropy inequalities, with error bounds relative to the entropy solution~\cite{deryck2024wpinns}.} Related work identifies nonuniqueness arising from finite information~\cite{langer2026nonuniqueness}, error control from sampled data~\cite{bonito2026consistent}, and parameter escape despite coercivity of a function-space energy~\cite{zuazua2026coercivity}. The distinction between interior and boundary contributions to gradients of integrals with parameter-dependent discontinuities also appears in Monte Carlo gradient estimation~\cite{lee2018reparameterization}.

Our contribution is a unified analysis of the implications needed to pass from residual optimization to the intended solution of the differential equation. Classical range and adjoint characterizations identify what stationarity tests. We follow how this information depends on the admissible boundary variations, the parameterization, and the regularity required to relate pointwise differentiation to the continuous first variation. We then examine whether minimizing sequences have convergent subsequences whose limits retain the required regularity, and whether these limits represent the intended solution. These questions are related, but their obstructions are distinct: a range condition addresses stationarity, whereas attainment and solution selection require further arguments.

The main results develop these connections in several settings:
\begin{itemize}
    \item For linear equations, we obtain an exact projection characterization
    of critical residuals under a compatibility condition on the
    parameterization of the trial functions. This characterization makes
    explicit what stationarity implies about the residual.

    \item {For stationary scalar conservation laws, minimizing the strong
    residual can fail to recover an entropy solution even in the limit.
    With a strictly convex flux and transonic boundary data, every
    minimizing sequence converges uniformly to the same continuous
    profile, which lies outside $H^1$ and does not solve the equation.
    This remains true when the boundary fluxes agree and an
    entropy-admissible stationary shock exists. The infimum is strictly
    positive and is not attained. The loss of regularity forces parameter
    escape whenever bounded parameter sets give bounded $H^1$ seminorms.
    For the corresponding equation with fixed positive viscosity,
    bounds in terms of the residual instead yield strong $H^2$ convergence
    to the solution as the residual tends to zero.}

    \item The fully nonlinear examples address both stationarity and
    solution admissibility. The Monge--Amp\`ere analysis in section \ref{subsec:monge-ampere} gives conditions
    under which stationarity forces the residual to vanish, while the
    Hamilton--Jacobi examples in section \ref{subsec:hj-ae} show why an additional admissibility
    condition can be necessary to identify the intended solution even
    when the residual is zero.
    \item {For piecewise-smooth trial families, we show that the regularity
    needed to define the residual functional does not by itself justify
    the sampled gradient used in training. We identify contributions from
    moving residual jumps that are absent from this gradient; the resulting
    bias is independent of the sample size. We also give a cellwise adjoint
    identity for smooth nonlinear second-order operators, including fully
    nonlinear equations. It describes the boundary and interface conditions
    on critical residuals, with a more explicit characterization for
    equations in divergence form. An explicit Poisson example shows that
    the sampled gradient can vanish identically while the continuous
    derivative is nonzero.}

\end{itemize}
Together, these results show why resolving one of these questions does not by itself resolve the others. More accurate sampling does not remove a function-space spurious critical point; approximation capacity does not determine the image of the admissible variations; and using the exact continuous gradient does not ensure that
minimizing sequences converge to an admissible solution. The favorable results specify conditions under which the corresponding implications do hold.

{Figure~\ref{fig:implication-chain} shows the chain of implications and where each link is established or broken.}

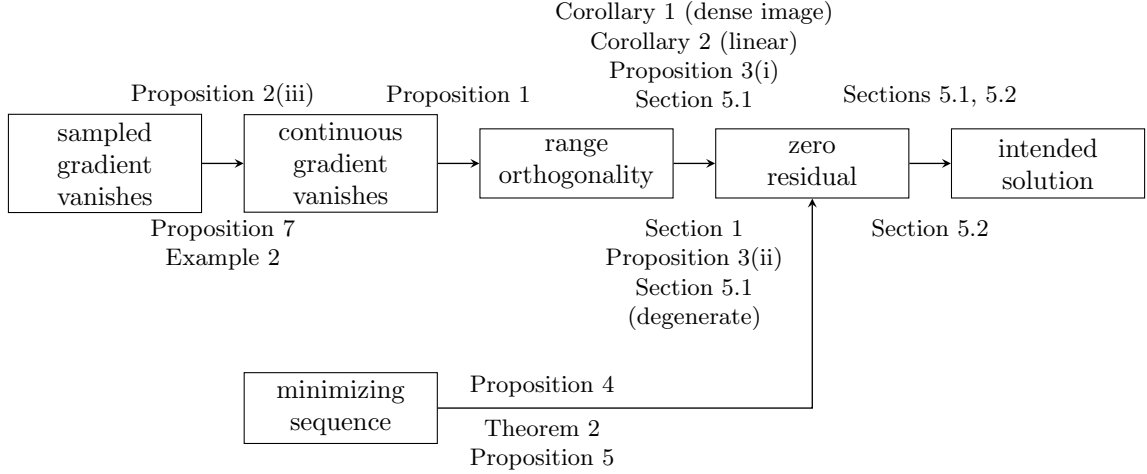
\begin{figure}[!htbp]
\centering
\captionsetup{font=footnotesize}
\begin{tikzpicture}[x=\linewidth, y=1pt,
  box/.style={draw, line width=0.4pt, inner sep=3pt, align=center,
              font=\footnotesize, text width=0.155\linewidth, minimum height=26pt},
  lab/.style={align=center, font=\scriptsize, inner sep=1pt},
  arr/.style={-stealth, line width=0.5pt}
]
\node[box] (b1) at (0.0875,0) {sampled\\gradient\\vanishes};
\node[box] (b2) at (0.29375,0) {continuous\\gradient\\vanishes};
\node[box] (b3) at (0.50,0) {range\\orthogonality};
\node[box] (b4) at (0.70625,0) {zero\\residual};
\node[box] (b5) at (0.9125,0) {intended\\solution};
\node[box] (b0) at (0.29375,-92) {minimizing\\sequence};
\draw[arr] (b1) -- (b2);
\draw[arr] (b2) -- (b3);
\draw[arr] (b3) -- (b4);
\draw[arr] (b4) -- (b5);
\draw[arr] (b0.east) -| (b4.south);
\node[lab, anchor=south] at (0.190625,20) {Proposition~\ref{prop:sampled-consistency}(iii)};
\node[lab, anchor=north] at (0.190625,-20) {Proposition~\ref{prop:sgd-bias}\\Example~\ref{ex:zero-sampled-gradient}};
\node[lab, anchor=south] at (0.396875,20) {Proposition~\ref{prop:parametric-fredholm}};
\node[lab, anchor=south] at (0.603125,20) {Corollary~\ref{cor:surjectivity} (dense image)\\Corollary~\ref{cor:linear-exactness} (linear)\\Proposition~\ref{prop:viscous-inviscid}(i)\\Section~\ref{subsec:monge-ampere}};
\node[lab, anchor=north] at (0.603125,-20) {Section~\ref{sec:intro}\\Proposition~\ref{prop:viscous-inviscid}(ii)\\Section~\ref{subsec:monge-ampere}\\(degenerate)};
\node[lab, anchor=south] at (0.809375,20) {Sections~\ref{subsec:monge-ampere}, \ref{subsec:hj-ae}};
\node[lab, anchor=north] at (0.809375,-20) {Section~\ref{subsec:hj-ae}};
\node[lab, anchor=south] at (0.47,-89) {Proposition~\ref{prop:viscous-confinement}};
\node[lab, anchor=north] at (0.47,-95) {Theorem~\ref{thm:burgers-cusp}\\Proposition~\ref{prop:reverse-transonic}};
\end{tikzpicture}
\caption{{The implications studied in this paper. Each box is a condition and each arrow a separate implication. Above an arrow: the results that establish it, each under its own hypotheses. Below: examples in which it fails. Corollary~\ref{cor:surjectivity} requires a dense image, which a finite-dimensional $\mathcal V_\theta$ cannot supply in $L^2$; the first arrow concerns the large-sample limit of the sampled gradient. The lower branch is the minimizing-sequence route of Section~\ref{sec:escape}, which bypasses stationarity.}}
\label{fig:implication-chain}
\end{figure}
\FloatBarrier

\section{Residual calculus and parameterized trial functions}
\label{sec:theory}

We first formulate the stationarity question for an exact residual functional. We then restrict the admissible variations to a parameterized family of trial functions (the hypothesis set) and ask whether differentiating a sampled loss recovers the resulting continuous first variation. The residual function, residual map, and residual space are distinguished because the differentiability needed for these steps depends on the chosen function spaces.

\subsection{Residual functions, residual maps, and residual spaces}
\label{subsec:tiers}

For some $k\in\mathbb{N}$ and a bounded domain $\Omega\subset\mathbb{R}^d$, let
$F\colon C^k(\overline{\Omega})\to C(\overline{\Omega})$ be a differential operator
which, to every sufficiently smooth function $u$, associates a continuous function
denoted by $x\mapsto F(x;u)$. We consider the equation $F(\cdot;u)=0$ in $\Omega$,
complemented with boundary conditions. For example, for a second-order equation on a
bounded interval $\Omega=(a,b)$, the differential operator is given by a continuous
function $\mathcal F\colon\overline{\Omega}\times\mathbb{R}^3\to\mathbb{R}$, so that
\begin{equation}\label{eq:F2}
  F(x;u)=\mathcal F\bigl(x,u(x),u'(x),u''(x)\bigr),\qquad
  \forall x\in\overline{\Omega},\ \forall u\in C^2(\overline{\Omega}).
\end{equation}

We write $DF(u)[v]$ for the linearization of $F$ at $u$ and $DF(u)^*$ for its formal adjoint. 
For $F$ given by \eqref{eq:F2} with $\mathcal F=\mathcal F(x,z,p,q)$ of class $C^1$, these are
\begin{equation}\label{eq:formal-linearization}
  DF(u)[v]=\mathcal F_q\,v''+\mathcal F_p\,v'+\mathcal F_z\,v,
  \qquad
  DF(u)^*r=(\mathcal F_q\,r)''-(\mathcal F_p\,r)'+\mathcal F_z\,r,
\end{equation}
with all coefficients evaluated at $(x,u(x),u'(x),u''(x))$.

The associated \emph{residual map}, denoted by $\mathcal{R}\colon X\to Y$, is given by
\[
  \mathcal{R}(u):=F(\cdot;u),\qquad \forall u\in X,
\]
where $X$ is a Banach space and $Y$ is a Hilbert space, chosen so that $F(\cdot;u)$
makes sense and belongs to $Y$ for every $u\in X$. If
$X\not\subset C^k(\overline{\Omega})$, then $F(\cdot;u)$ is defined a.e.\ by the same
formula, with weak derivatives. 

The choice of $X$ and $Y$ depends on the problem; for
instance, homogeneous Dirichlet conditions can be built into $X$. {When nonhomogeneous boundary data are imposed exactly, we let $X$ carry the homogeneous form of the boundary conditions and fix a function $u_b$ satisfying the data and having the regularity required of elements of $X$; the trial functions then range over the affine space $u_b+X$, the residual map is defined on $u_b+X$, and the condition $\mathcal{R}\in C^1(X;Y)$ below is understood for the map $w\mapsto\mathcal{R}(u_b+w)$. In both cases the admissible variations are the elements of $X$.} We call $Y$ the
\emph{residual space}, and the associated \emph{least-squares loss functional} is
given by
\[
  J(u):=\frac12\|\mathcal{R}(u)\|_Y^2 .
\]
{We call $r=\mathcal{R}(u)\in Y$ the \emph{residual function} of $u$.}
In the standard setting we consider $Y=L^2(\Omega)$, although other norms can be
considered, for instance by replacing the Lebesgue measure on $\Omega$ by a different
measure $\mu$. If $\mu$ is not absolutely continuous with respect to the Lebesgue
measure, e.g.\ an empirical measure on collocation points, then $F(\cdot;u)$ must be
defined $\mu$-a.e., which requires more regularity, for instance
$X\hookrightarrow C^k(\overline{\Omega})$.

Let us assume that $\mathcal{R}\in C^1(X;Y)$. This may be guaranteed
by a suitable choice of $X$ and $Y$. For instance, for the equation $F(\cdot;u)=0$
with $F$ given by \eqref{eq:F2}, homogeneous Dirichlet conditions and
$Y=L^2(\Omega)$: if $\mathcal F\in C^1(\overline{\Omega}\times\mathbb{R}^3)$, then $\mathcal{R}\in C^1(X;Y)$ is guaranteed by choosing
$X=W^{2,\infty}(\Omega)\cap H^1_0(\Omega)$. If moreover $\mathcal F$ is quasilinear, i.e.\
affine with respect to $u''$, one can simply take
$X=H^2(\Omega)\cap H^1_0(\Omega)$.

For any $u\in X$, we denote by $D\mathcal{R}(u)\in\mathcal{L}(X,Y)$ the Fr\'echet
derivative of $\mathcal{R}$ at $u$, and by $D\mathcal{R}(u)^*\in\mathcal{L}(Y,X')$
its adjoint, defined by
$$
\langle D\mathcal{R}(u)^*r,v\rangle_{X',X}:=\langle r,D\mathcal{R}(u)[v]\rangle_Y \quad \text{for all $r\in Y$
and $v\in X$.}
$$
{When $F$ is of the form \eqref{eq:F2} and $\mathcal{R}\in C^1(X;Y)$, the Fr\'echet derivative is given by the linearization \eqref{eq:formal-linearization}: $D\mathcal{R}(u)[v]=DF(u)[v]$ for all $v\in X$.} By the chain rule, the loss functional $u\mapsto J(u)$ is Fr\'echet differentiable, with
\[
  DJ(u)[v]=\bigl\langle\mathcal{R}(u),D\mathcal{R}(u)[v]\bigr\rangle_Y
  =\bigl\langle D\mathcal{R}(u)^*\mathcal{R}(u),v\bigr\rangle_{X',X},
  \qquad\forall v\in X.
\]
Hence, in the unconstrained optimization framework, $\bar u\in X$ is a critical
point of $J$ if and only if
\begin{equation}\label{eq:crit}
  \mathcal{R}(\bar u)\in
  \bigl(D\mathcal{R}(\bar u)[X]\bigr)^{\perp}
  =\ker D\mathcal{R}(\bar u)^*,
\end{equation}
where $D\mathcal{R}(\bar u)[X]$ denotes the range of $D\mathcal{R}(\bar u)$ and
$\perp$ the orthogonal complement in $Y$. Condition \eqref{eq:crit} gives two equivalent routes to address both
questions Q1 and Q2 presented in the introduction.
For  Q1, the primal route asks whether the linearized test problem
$D\mathcal R(\bar u)[v]=g$, with $v\in X$, is solvable for a dense set
of $g\in Y$. The equivalent dual route asks whether the homogeneous
adjoint problem $D\mathcal R(\bar u)^*r=0$, with $r\in Y$, has only
the trivial solution, that is, $\ker D\mathcal R(\bar u)^*=\{0\}$.
Either condition implies $\mathcal R(\bar u)=0$ by \eqref{eq:crit}.
If this holds at every admissible $u$, then every critical point of $J$
solves the equation.

Regarding Q2, a nonzero residual at a critical point $\bar u$
must be a nontrivial solution $r\in Y$ of the linear equation
$D\mathcal{R}(\bar u)^*r=0$. Since $\bar u$ is itself unknown, this is a necessary
condition rather than a characterization, and in general not every element of
$\ker D\mathcal{R}(\bar u)^*$ occurs as a residual.

{Boundary conditions can also be included through a least-squares
penalty, rather than imposing them into the choice of $X$. In this case, choose $X$ without imposing boundary conditions and
assume $\mathcal R\in C^1(X;Y)$. Let $\mathcal B\in C^1(X;Z)$ be
the boundary operator, with data $g\in Z$ in a Hilbert space $Z$.
For a fixed $\beta>0$, equip
$\widetilde Y:=Y\times Z$ with the weighted inner product
\[
\langle(r,b),(s,c)\rangle_{\widetilde Y}
:=\langle r,s\rangle_Y+\beta\langle b,c\rangle_Z.
\]
The augmented residual
$\widetilde{\mathcal R}(u):=(\mathcal R(u),\mathcal B(u)-g)$ then gives
\[
J_\beta(u)
=\frac12\|\widetilde{\mathcal R}(u)\|_{\widetilde Y}^2
=\frac12\|\mathcal R(u)\|_Y^2
+\frac{\beta}{2}\|\mathcal B(u)-g\|_Z^2.
\]
The chain rule and characterization~\eqref{eq:crit} apply to
$\widetilde{\mathcal R}$, with orthogonality and the adjoint taken in
this weighted product space. Stationarity now requires the sum of the
interior and boundary pairings to vanish for every admissible variation;
it does not separately force each pairing to vanish.}

In practice, however, two main obstacles have to be taken into account. On one hand, the search for the solution is constrained to a parameterized family of functions $\mathcal{U} \subset X$, typically with a finite number of parameters. On the other hand, the functional $J(u)$ or its gradient cannot be computed exactly; instead, it is only estimated by random sampling.

\subsection{Parameterized functions and admissible variations}
\label{subsec:parametric-functions}
Let us focus on the first obstacle, which is the parameterization constraint.
For $N\in \N$ being the number of parameters, we consider a $C^1$ map
$$
\theta \in \R^N \longmapsto u_\theta(\cdot) \in X
$$
{(into $u_b+X$ when nonhomogeneous boundary data are imposed exactly).}
This means in particular that the partial derivative $w_k:=\partial_{\theta_k} u_\theta(\cdot)$ belongs to $X$, for each parameter $\theta_k$.
The optimization problem is then constrained to the family of functions
\begin{equation}
\label{space of parametrized functions}
\mathcal{U} := \left\{ u_\theta (\cdot)  \quad : \quad \theta = (\theta_k)_{k=1}^N \in \mathbb{R}^{N} \right\}.
\end{equation}
In the deep learning literature, $\mathcal{U}$ is often referred to as \emph{the hypothesis set}.

For any $\theta \in \R^N$, we define \emph{the space of admissible variations} as
\begin{equation}
\label{admissible variations}
\mathcal{V}_\theta := \operatorname{span} \left\{ w_k=\partial_{\theta_k} u_\theta (\cdot) \ : \quad k = 1, \ldots, N \right\} \subset X
\end{equation}
Note that the linear space $\mathcal{V}_\theta\subset X$ is well defined due to the differentiability of the map $\theta \mapsto u_\theta$.

In the particular case of a linear parameterization, where $u_\theta := \sum_{k=1}^N \theta_k \phi_k$ for a given finite set of functions $\{\phi_k\}_{k=1}^N\subset X$, the space of admissible variations is independent of the parameter $\theta$, and given by $\mathcal{V}_\theta = \operatorname{span} \{ \phi_k \ : \ k=1,\ldots, N \}$ for every $\theta$.
In the neural network framework, where $\theta \mapsto u_\theta$ is non-linear, the space of admissible variations varies with the parameter $\theta$. Note that $\mathcal{V}_\theta$ is a feature of the parameter, not of the function $u_\theta$. It might happen that for two parameters $\theta\neq \theta'$, the associated function is the same ($u_\theta = u_{\theta'}$), while the corresponding spaces of admissible variations are different ($\mathcal{V}_\theta \neq \mathcal{V}_{\theta'}$).

For the least-squares loss functional $J(u_\theta)= \frac{1}{2}\| \mathcal{R}(u_\theta) \|_Y^2$, we have the following critical point condition, relating the residual $\mathcal{R}(u_\theta)$ and the range of the linearization $D\mathcal{R}(u_\theta)$, restricted to the space of admissible variations.

\begin{proposition}[Parametric range orthogonality]\label{prop:parametric-fredholm}
Let $X$ be a Banach space and $Y$ a Hilbert space, and let $\mathcal{R}\in C^1(X; Y)$ be a given residual map. Consider the loss functional $J(u) = \frac{1}{2}\| \mathcal{R}(u) \|_Y^2$. Let $\mathcal{U}\subset X$ be a parameterized family of functions as in \eqref{space of parametrized functions}, and for each $\theta\in \R^N$, let $\mathcal{V}_\theta$ be the space of admissible variations defined in \eqref{admissible variations}. If \(\theta^*\) is a
critical point of \(\theta \mapsto J(u_\theta)\), then
\begin{equation}\label{eq:parametric-range-orthogonality}
\mathcal R(u_{\theta^*})
\in
\bigl(D\mathcal R(u_{\theta^*})[\mathcal V_{\theta^*}]\bigr)^\perp.
\end{equation}
\end{proposition}

\begin{proof}
Since $\theta^\ast$ is a critical point, we have $\partial_{\theta_k} J(u_{\theta^\ast}) = 0$ for all $k =1,\ldots, N$.
By the chain rule, the partial derivative of $J(u_\theta)$ with respect to each parameter $\theta_k$ is given by
$$
\partial_{\theta_k} J(u_\theta)  =  \bigl\langle \mathcal R(u_\theta),D\mathcal R(u_\theta)[w_k]\bigr\rangle_Y = DJ (u_\theta) [w_k], \qquad \text{where} \ w_k = \partial_{\theta_k} u_\theta (\cdot).
$$
Then, in view of the definition of the linear space $\mathcal{V}_\theta$ in \eqref{admissible variations}, at any critical $\theta^\ast$ we have
\begin{equation}\label{eq:dJ_without_integration_by_parts}
DJ({u}_{\theta^\ast})[v]
=
\bigl\langle \mathcal R({u}_{\theta^\ast}),D\mathcal R({u}_{\theta^\ast})[v]\bigr\rangle_Y = 0,
\qquad \forall  v\in\mathcal V_{\theta^\ast}.
\end{equation}
In other words, at a critical point $\theta^\ast$, it holds that
\begin{equation}\label{eq:critical-point-basic}
\mathcal R({u}_{\theta^\ast})\perp D\mathcal R({u}_{\theta^\ast})[\mathcal V_{\theta^\ast}].
\end{equation}
\end{proof}

In view of this, we can give the following criterion to answer Q1, based on the range of the residual's differential restricted to the space of admissible variations.

\begin{corollary}[Dense-range criterion]\label{cor:surjectivity}
Under the assumptions of Proposition \ref{prop:parametric-fredholm},
if \(\theta^\ast\) is a critical point of \(\theta \mapsto J(u_\theta)\) and
\[
\overline{D\mathcal R({u}_{\theta^\ast})[\mathcal V_{\theta^\ast}]}=Y,
\]
then \(\mathcal R({u}_{\theta^\ast})=0\).
\end{corollary}

{This criterion comes with an important caveat concerning the dimensionality. If one considers a finite number of parameters, the space $\mathcal{V}_\theta$ is finite-dimensional, and so is $D\mathcal R(u_\theta)[\mathcal V_\theta]$. However, if the chosen residual space is $Y= L^2(\Omega)$, as it is the standard choice in PINNs, the above density condition cannot hold due to $Y$ being infinite-dimensional. Hence, Corollary~\ref{cor:surjectivity} cannot be applied unless one considers a parameterized class of functions with an infinite number of parameters, or one reduces the residual space $Y$ to be finite dimensional.}

Note however that, in general, the finite number of parameters is not the only obstruction preventing the residual from vanishing.
The motivating Neumann problem \eqref{eq:neumann-semilinear} is precisely a failure of
{the full-space range condition discussed after \eqref{eq:crit}}, where the variations range over all of $X$. At \(u=0\), \(D\mathcal R(0)[v]=-v''\), and the
Neumann condition in the definition of $X$ forces that, for every $v\in X$, the function $g =-v''$
has mean zero. Therefore, any constant residual $r$ lies in the
orthogonal complement of the range of $D\mathcal{R}(0)$. In particular, $\mathcal{R}(0) = -1$ satisfies $\mathcal{R}(0) \perp D\mathcal{R}(0) [X]$.
This example shows how the boundary conditions restrict what stationarity can test.

If $\mathcal R$ is twice Fr\'echet differentiable, then we have
\begin{equation}\label{eq:hessian-residual-loss}
D^2J(u_\theta)[v,v]
=
\|D\mathcal R(u_\theta)[v]\|_Y^2
+
\bigl\langle \mathcal R(u_\theta),D^2\mathcal R(u_\theta)[v,v]\bigr\rangle_Y,
\end{equation}
along any admissible variation $v\in \mathcal{V}_\theta$.
The second term is sign-indefinite in general, so nonlinear residual maps can produce nonconvex least-squares functionals. None of the range arguments below uses \eqref{eq:hessian-residual-loss}; they rely only on the first-order identity \eqref{eq:dJ_without_integration_by_parts} and the geometry of $D\mathcal R(u_\theta)[\mathcal V_\theta]$.

An equivalent criticality condition can be obtained using the adjoint of $D\mathcal{R}({u}_\theta)$.
Write 
$$
T_\theta:=D\mathcal R({u}_\theta)|_{\mathcal V_\theta}:\mathcal V_\theta\to Y
$$ 
for the linearized
residual map restricted to the set of admissible variations $\mathcal{V}_\theta$, and
$T_\theta^*:Y\to\mathcal V_\theta'$ for its adjoint, so that
\[
\ker T_\theta^*=\{r\in Y:\langle r,T_\theta v\rangle_Y=0\text{ for all }v\in\mathcal V_\theta\}.
\]
Condition \eqref{eq:critical-point-basic} states that, at a critical point $\theta^*$, the residual
$r:=\mathcal R(u_{\theta^*})$ satisfies
\[
r\in(\operatorname{Ran}T_{\theta^*})^\perp=\ker T_{\theta^*}^*.
\]
The restricted adjoint $T_\theta^*r$ is a linear functional on
$\mathcal V_\theta$. When the regularity permits integration by parts,
its action on a variation is the sum of a volume pairing with the
formal adjoint $DF(u_\theta)^*r$ and boundary or interface pairings;
see Proposition~\ref{prop:piecewise-adjoint}. For finitely many parameter
directions, these contributions can cancel against each other within each parameter direction, so
$T_\theta^*r=0$ does not by itself imply the local adjoint equation or
separate boundary and interface conditions. The local equation follows
when the variation space contains arbitrary smooth tests compactly
supported in each cell. The remaining trace conditions follow from the
traces that can be varied independently, as in
Section~\ref{subsec:piecewise-setup}.

The primal and dual criteria for Q1 in Section~\ref{subsec:tiers}
apply here to the restricted operator $T_\theta$. 
For Q2, every critical residual belongs to $\ker T_\theta^*$,
but a nonzero element of this null space is not necessarily
realized as $\mathcal R(u_\theta)$ at a critical parameter.

A restricted space of admissible variations ($\mathcal V_{\theta^*}\subsetneq X$) does not automatically imply a weaker stationarity condition. The omitted variation directions might simply map to zero or duplicate directions already generated by $\mathcal V_{\theta^*}$. What truly matters is the resulting image in the residual space. This issue is distinct from approximation capacity. Even if a parameterized
family can approximate the exact solution arbitrarily well, the residual
can be nonzero at a parameter value where all parameter derivatives vanish.

\subsection{Sampled gradient}
\label{subsec:sampled-gradient}

{
Here we assume that the trial functions $u_\theta$ are sufficiently smooth so that the differential operator $F(\cdot; u_\theta)$ makes sense pointwise.
Write $r(x,\theta):=F(x;u_\theta)$ for the pointwise evaluation of the residual of the trial function $u_\theta$ at $x$. For $n$ points $\xi_1,\ldots,\xi_n$ sampled independently and uniformly from $\Omega$, the Monte Carlo approximation of the continuous residual functional is 
\begin{equation}\label{eq:sampled-loss}
 \hat J(\theta;\xi):=\frac{|\Omega|}{n}\sum_{j=1}^{n}\tfrac12\,r(\xi_j,\theta)^2,
 \end{equation}  and, whenever it exists, its parameter gradient with the sample locations held fixed
\begin{equation}\label{eq:sampled-gradient}
\nabla_\theta\hat J(\theta;\xi)
=\frac{|\Omega|}{n}\sum_{j=1}^{n} r(\xi_j,\theta)\,\nabla_\theta r(\xi_j,\theta),
\end{equation}
is referred to as the \emph{sampled gradient}.
The following proposition records the standard results on what the two estimators converge to, when the integrands are sufficiently smooth.

\begin{proposition}[Large-sample limits of the sampled loss and gradient]\label{prop:sampled-consistency}
Let $\Omega\subset\mathbb R^d$ be bounded, let $Y=L^2(\Omega)$, and let
$\xi_1,\dots,\xi_n$ be independent and uniformly distributed on $\Omega$. The following hold at each fixed parameter value $\theta$.
\begin{enumerate}
\item If $r(\cdot,\theta)\in L^2(\Omega)$, then
$\mathbb E[\hat J(\theta;\xi)]=J(\theta)$ and $\hat J(\theta;\xi)\to J(\theta)$
almost surely as $n\to\infty$.
\item If for almost every $x\in\Omega$, $r(x,\theta)$
is differentiable at $\theta$, and
$r(\cdot,\theta)\,\nabla_\theta r(\cdot,\theta)\in L^1(\Omega)$, then almost
surely $\hat J(\cdot\,;\xi)$ is differentiable at $\theta$, with
\begin{equation}\label{eq:sampled-gradient-mean}
\mathbb E\bigl[\nabla_\theta\hat J(\theta;\xi)\bigr]
=G(\theta):=\int_\Omega r(x,\theta)\,\nabla_\theta r(x,\theta)\,dx,
\end{equation}
and $\nabla_\theta\hat J(\theta;\xi)\to G(\theta)$ almost surely as
$n\to\infty$.
\item  If, in addition to the hypotheses of (ii), the map
$\vartheta\mapsto \mathcal R(u_\vartheta)=r(\cdot,\vartheta)$ is Fr\'echet differentiable at
$\theta$ as a $Y$-valued map, then
$G(\theta)=\nabla_\theta J(\theta)$.
\end{enumerate}
\end{proposition}

\begin{proof}
 Parts (i) and (ii) follow from the strong law of large numbers and the integrability assumptions. For (iii), Fr\'echet differentiability gives convergence of the parameter difference quotients in $L^2(\Omega)$. Along a subsequence they converge almost everywhere, so hypothesis (ii) identifies their limits with the pointwise derivatives $\partial_{\theta_k}r$. Differentiating $J(\theta)=\tfrac12\|\mathcal R(u_\theta)\|_{Y}^2$ therefore gives $\partial_{\theta_k}J(\theta)=\langle r(\cdot,\theta),\partial_{\theta_k}r(\cdot,\theta)\rangle_{L^2}=G_k(\theta)$ for all $k$.
 \end{proof}

Thus the sampled loss consistently estimates $J$ whenever the residual is
square-integrable, while the sampled gradient converges to {the mean sampled gradient $G(\theta)$}.  The hypothesis of (iii) holds, for example, under the standing assumptions of Sections~\ref{subsec:tiers} and~\ref{subsec:parametric-functions}. In that case the chain rule yields $G(\theta)=\nabla_\theta J(\theta)$. {For moving-interface parameterizations, residual jumps at
parameter-dependent locations can prevent the residual map from
being differentiable into $L^2(\Omega)$. Differentiating $J$ can
then produce interface-motion terms absent from $G$.
Proposition~\ref{prop:sgd-bias} computes this difference.
 The same discrepancy occurs for a deterministic quadrature rule that converges for the integrand $r\,\partial_{\theta_k}r$: the quadrature approximation of this pointwise derivative converges to $G_k$, {which can differ from $\partial_{\theta_k}J$.} Random sampling additionally introduces variance.

 For each fixed sample set, the sampled gradient is the exact derivative of the sampled functional wherever that derivative exists. The consistency question concerns its relation to the derivative of the continuous functional as the sample size increases. Parameter-dependent interfaces can invalidate this consistency even when the sampled functional itself is a consistent approximation of the continuous loss.
}
}

\section{Linear equations: projection geometry}
\label{sec:linear-pdes}

Let $X$ and $Y$ be as in Section~\ref{subsec:tiers}, and let
$A:X\to Y$ be bounded and linear. {Fix a function $u_b$ satisfying the prescribed
boundary data, taking $u_b=0$ for homogeneous data. We assume that $A$
extends linearly to $X+\operatorname{span}\{u_b\}$, with values in $Y$,
and use the same symbol for this extension. Consider $Au=f$, with
$f\in Y$ and $u\in u_b+X$. The residual map
$\mathcal R(u)=F(\cdot;u)=Au-f$ satisfies
\[
\mathcal R(u_b+w)=Aw+Au_b-f,\qquad w\in X.
\]
Thus the shifted map $w\mapsto\mathcal R(u_b+w)$ is $C^1$ from $X$ to
$Y$, and $D\mathcal R(u)[v]=Av$ for $u\in u_b+X$ and $v\in X$.}

{In what follows we write $J(\theta):=J(u_\theta)$ for the loss as a function of the parameters.}
The loss functional is
\[
J(\theta)=\frac12\|Au_\theta-f\|_{Y}^2,
\]
where {$\mathcal U=\{u_\theta:\theta\in\mathbb R^N\}\subset u_b+X$
is a parameterized family as in \eqref{space of parametrized functions},
with $\theta\mapsto u_\theta$ of class $C^1$ in the affine-space convention
of Section~\ref{subsec:tiers}. We write}
\[
 u_\theta=u_b+\widetilde u_\theta,
\]
with $\widetilde u_\theta\in X$ the homogeneous component, and we impose the compatibility condition
\begin{equation}\label{eq:trial-space-compatibility-linear}
\widetilde u_\theta\in\mathcal V_\theta.
\end{equation}
This compatibility condition holds automatically for the linear model \eqref{eq:linear-trial-model} below. It also holds for many neural parameterizations with a final linear layer, after separating the fixed function used to satisfy the boundary data.

{The image $A\mathcal V_\theta=D\mathcal R(u_\theta)[\mathcal V_\theta]\subset Y$ is the space appearing in Proposition~\ref{prop:parametric-fredholm}; it is closed because $\mathcal V_\theta$ is finite-dimensional.}

\begin{theorem}[Projection formula for a linear equation]\label{thm:linear-projection}
Assume \eqref{eq:trial-space-compatibility-linear}, and let $\theta^*$ be a critical point of $J$. Then
\begin{equation}\label{eq:linear-projection}
A(u_{\theta^*}-u_b)
=
P_{A\mathcal V_{\theta^*}}(f-Au_b),
\end{equation}
and
\begin{equation}\label{eq:linear-residual-distance}
\|Au_{\theta^*}-f\|_{Y}
=
\operatorname{dist}_{Y}
\bigl(f-Au_b,A\mathcal V_{\theta^*}\bigr),
\end{equation}
where $P_{A\mathcal{V}_{\theta^\ast}}$ is the orthognal projection of $Y$ onto $A\mathcal{V}_{\theta^\ast}$.
\end{theorem}

\begin{proof}
Criticality and Proposition~\ref{prop:parametric-fredholm} give
$r:=Au_{\theta^*}-f\perp A\mathcal V_{\theta^*}$. By
\eqref{eq:trial-space-compatibility-linear},
$A(u_{\theta^*}-u_b)\in A\mathcal V_{\theta^*}$. Since
\[
f-Au_b=A(u_{\theta^*}-u_b)-r,
\]
this is the orthogonal decomposition of $f-Au_b$ into its
$A\mathcal V_{\theta^*}$ and $(A\mathcal V_{\theta^*})^\perp$ components. Equations
\eqref{eq:linear-projection} and \eqref{eq:linear-residual-distance} follow.
\end{proof}

\begin{corollary}[Exactness criterion]\label{cor:linear-exactness}
Under the hypotheses of Theorem~\ref{thm:linear-projection},
\[
Au_{\theta^*}=f
\quad\Longleftrightarrow\quad
f-Au_b\in A\mathcal V_{\theta^*}.
\]
\end{corollary}

\begin{corollary}[A posteriori estimate]\label{cor:a-posteriori}
Let $u^*\in u_b+X$ solve $Au=f$ with the same boundary data as $u_{\theta^*}$, so that $u_{\theta^*}-u^*\in X$. Assume that, for some norm $\|\cdot\|_E$ on $X$ and some constant $C_A>0$,
\[
\|v\|_E\le C_A\|Av\|_{Y}
\qquad\text{for every }v\in X.
\]
Then every critical point covered by Theorem~\ref{thm:linear-projection} satisfies
\[
\|u_{\theta^*}-u^*\|_E
\le
C_A
\operatorname{dist}_{Y}
\bigl(f-Au_b,A\mathcal V_{\theta^*}\bigr).
\]
\end{corollary}

\begin{proof}
The difference $u_{\theta^*}-u^*$ belongs to $X$, so the stability estimate and \eqref{eq:linear-residual-distance} give
\[
\|u_{\theta^*}-u^*\|_E
\le C_A\|Au_{\theta^*}-f\|_{Y}
= C_A\operatorname{dist}_{Y}\bigl(f-Au_b,A\mathcal V_{\theta^*}\bigr).
\]
\end{proof}

The projection distance in this estimate is computable after training from the finite-dimensional space
\[
A\mathcal V_{\theta^*}
=
\operatorname{span}\{A\partial_{\theta_k}u_{\theta^*}:k=1,\ldots,N\},
\]
{so the whole bound is computable whenever a usable stability constant $C_A$ is available.} By Corollary~\ref{cor:linear-exactness} it vanishes exactly when \(f-Au_b\) lies in \(A\mathcal V_{\theta^*}\); we do not require the much stronger condition that $u^*-u_b$ belong to $\mathcal V_\theta$ for every $\theta$.
{The corollary reads the continuum residual as an a posteriori error indicator, with $C_A$ as the conversion factor. When the only information about $f$ and the boundary data consists of point values at finitely many sites, \cite{bonito2026consistent} show that {the standard discrete $L^2$ residual loss can fail to control the error at the rate permitted by this information} and construct losses that do; their analysis concerns minimizers of the loss and leaves the optimization aside, which is the part addressed here.}

For the Dirichlet Laplacian, take $E=L^2(\Omega)$ and use the spectral stability estimate
\[
\|v\|_{L^2(\Omega)}\le\lambda_1^{-1}\|{-\Delta v}\|_{L^2(\Omega)}
\]
for homogeneous Dirichlet errors. Corollary~\ref{cor:a-posteriori} then gives
\[
\|u_{\theta^*}-u^*\|_{L^2(\Omega)}
\le
\lambda_1^{-1}
\operatorname{dist}_{L^2(\Omega)}
\bigl(f-Au_b,A\mathcal V_{\theta^*}\bigr),
\]
where $\lambda_1$ is the first Dirichlet eigenvalue.

\medskip\noindent\textbf{Approximate stationarity.}
The same decomposition distinguishes an inadequate image $A\mathcal V_\theta$ from an unresolved optimization error. At any parameter value satisfying \eqref{eq:trial-space-compatibility-linear}, define the Jacobian of the map $\theta\mapsto \mathcal{R}(u_\theta) = Au_\theta - f$, denoted by $B_\theta: \R^N\to Y$, as
\[
B_\theta c:=\sum_{k=1}^N c_kAw_k, \qquad \forall c = (c_1, \ldots, c_N)\in \R^N.
\]
Its adjoint $B_\theta^\ast: Y\to \R^N$ is given by $(B_\theta^\ast r)_k = \langle Aw_k, r\rangle_Y$ for $k = 1, \ldots, N$.
We also define
\[
\qquad H_\theta:=B_\theta^*B_\theta,
\quad \text{and} \quad g_\theta:=\nabla_\theta J(\theta)=B_\theta^*(Au_\theta-f).
\]
Let $H_\theta^\dagger$ denote the Moore--Penrose inverse. Since
$B_\theta H_\theta^\dagger B_\theta^*=P_{A\mathcal V_\theta}$, orthogonal decomposition gives
\begin{equation}\label{eq:approximate-projection}
\|Au_\theta-f\|_{Y}^2
=\operatorname{dist}_{Y}(f-Au_b,A\mathcal V_\theta)^2
+g_\theta^T H_\theta^\dagger g_\theta.
\end{equation}
Indeed, the perpendicular component of the residual is $-P_{(A\mathcal V_\theta)^\perp}(f-Au_b)$, while its component in $A\mathcal V_\theta$ has squared norm $g_\theta^TH_\theta^\dagger g_\theta$. Under the stability estimate of Corollary~\ref{cor:a-posteriori}, the square root of the right-hand side, multiplied by $C_A$, bounds the solution error. A small Euclidean parameter gradient alone is therefore insufficient unless the conditioning of the nonzero directions $Aw_k$ is also controlled. 

\medskip\noindent\textbf{Linear trial spaces.}
A particularly simple special case occurs when the parameterization itself
is linear,
\begin{equation}\label{eq:linear-trial-model}
u_\theta=u_b+\sum_{k=1}^N\theta_k\phi_k,
\end{equation}
with a fixed function $u_b$ carrying the boundary data and fixed functions
$\phi_k$. Then
\[
\mathcal V_\theta=\operatorname{span}\{\phi_1,\ldots,\phi_N\}
\]
is independent of $\theta$, and so is $A\mathcal V_\theta$.
Theorem~\ref{thm:linear-projection} therefore reduces to ordinary
least-squares projection onto the fixed space
$A\operatorname{span}\{\phi_1,\ldots,\phi_N\}$: a positive residual is the
component of the data not represented by the image of the trial space, and
the parameter optimization is therefore a convex quadratic least-squares problem. Sufficiently
smooth spline and spectral spaces are the classical examples; for a
second-order strong residual the $\phi_k$ must have enough regularity for
$\phi_k''$ to lie in the residual space. Neural parameterizations, by
contrast, generally have parameter-dependent spaces of admissible variations.

\section{Scalar conservation laws}
\label{sec:escape}

Equations in divergence form include important classes of elliptic
equations and hyperbolic conservation laws. In this section, we
consider one-dimensional scalar conservation laws and examine how
viscosity affects residual minimization. We first determine what
stationarity implies about the residual under Dirichlet boundary
conditions. We then study the limits of minimizing sequences.
In the inviscid problem, these limits can lose the regularity
required of admissible trial functions, leading to parameter escape.

\subsection{The effect of viscosity}
\label{subsec:viscous-inviscid}

Consider the one-dimensional family
\begin{equation}\label{def:F_eps}
    F^\varepsilon(x;u):=\partial_x f\bigl(u(x)\bigr)-\varepsilon u_{xx}(x)-g(x),
\qquad x\in(0,1),
\end{equation}
with $f\in C^2(\mathbb R)$ and {$g\in L^2(0,1)$}. We impose Dirichlet boundary conditions
\begin{equation}\label{def:u_l-u_r}
u(0)=u_\ell,
\qquad
u(1)=u_r,    
\end{equation}
and least-squares functional
\begin{equation}\label{def:J_eps}
J^\varepsilon(u):=\frac12\int_0^1 |F^\varepsilon(x;u)|^2\,dx.
\end{equation}
{This is the loss functional of Section~\ref{subsec:tiers} for the residual map $\mathcal R^\varepsilon(u):=F^\varepsilon(\cdot;u)$, with $Y=L^2(0,1)$ and with $X$ the space carrying the homogeneous Dirichlet conditions, $H^2(0,1)\cap H^1_0(0,1)$ for $\varepsilon>0$ and $H^1_0(0,1)$ for $\varepsilon=0$; since $H^1(0,1)\hookrightarrow C^0([0,1])$ and $f\in C^2(\mathbb R)$, $\mathcal R^\varepsilon\in C^1(X;Y)$ in both cases.}

\begin{proposition}[Consequences of viscosity on stationarity]
\label{prop:viscous-inviscid}
Let $F^\varepsilon$, $J^\varepsilon$, and the Dirichlet data be as above,
and suppose that $\bar u$ satisfies these boundary conditions.
\begin{enumerate}
\item Let $\varepsilon>0$ and
$\bar u\in H^2(0,1)$. If
\[
DJ^\varepsilon(\bar u)[v]=0
\qquad\text{for every }v\in H^2(0,1)\cap H^1_0(0,1),
\]
then $F^\varepsilon(\cdot;\bar u)=0$ almost everywhere on $(0,1)$.\\

\item Let $\varepsilon=0$ and
$\bar u\in H^1(0,1)$, and assume that
$|f'(\bar u)|>0$ on $[0,1]$.
Then $DJ^0(\bar u)[v]=0$ for every $v\in H^1_0(0,1)$ if and only if
\[
\partial_x f(\bar u)-g=c~\text{almost everywhere},
\qquad
c=f(u_r)-f(u_\ell)-\int_0^1g(x)\,dx.
\]
Consequently, a stationary $\bar u$ solves the differential equation
if and only if
\[
f(u_r)-f(u_\ell)=\int_0^1g(x)\,dx.
\]
\end{enumerate}
\end{proposition}

\begin{proof}
Stationarity is condition \eqref{eq:crit}, $\mathcal R^\varepsilon(\bar u)\perp D\mathcal R^\varepsilon(\bar u)[X]$,
with $X=H^2(0,1)\cap H^1_0(0,1)$ in part (i) and $X=H^1_0(0,1)$ in part (ii).
We determine the range $D\mathcal R^\varepsilon(\bar u)[X]$
by solving
\begin{equation}\label{eq:test-problem}
\partial_x\bigl(f'(\bar u)v-\varepsilon v_x\bigr)=\tilde g,
\qquad v\in X.
\end{equation}
For $\tilde g\in L^2(0,1)$, set
$\widetilde G(x):=\int_0^x\tilde g(s)\,ds$.
One integration gives
\[
f'(\bar u)v-\varepsilon v_x=\widetilde G+C.
\]
A direct calculation gives the candidate solution
\[
v(x)=
\begin{cases}
-\displaystyle\frac{1}{\varepsilon\mu_\varepsilon(x)}
\int_0^x\mu_\varepsilon(s)\bigl(\widetilde G(s)+C\bigr)\,ds,
& \varepsilon>0,\\[15pt]
f'(\bar u(x))^{-1}(\widetilde G(x)+C),
& \varepsilon=0,
\end{cases}
\]
where, for $\varepsilon>0$,
\[
\mu_\varepsilon(x)
:=\exp\!\left(-\frac1\varepsilon
\int_0^x f'(\bar u(s))\,ds\right) >0.
\]

\noindent\textup{(i)} In the viscous case, $v(0)=0$ already holds, and
$v(1)=0$ determines
\[
C=-\frac{\int_0^1\mu_\varepsilon(s)\widetilde G(s)\,ds}
        {\int_0^1\mu_\varepsilon(s)\,ds}.
\]
Since
$f'(\bar u)\in C^1([0,1])$ and $\widetilde G\in H^1(0,1)$, the
resulting function belongs to $H^2(0,1)\cap H^1_0(0,1)$.
Thus every $\tilde g\in L^2(0,1)$ is attainable, and
{the full-space orthogonality condition \eqref{eq:crit} implies}
$F^\varepsilon(\cdot;\bar u)=0$.

\noindent\textup{(ii)} In the inviscid case, $v(0)=0$ implies $C=0$, while
$v(1)=0$ requires
\[
\widetilde G(1)=\int_0^1\tilde g(x)\,dx=0.
\]
Conversely, for every $\tilde g$ with zero mean, the nonsonic
assumption ensures that $v=\widetilde G/f'(\bar u)$ belongs to
$H^1_0(0,1)$ and solves \eqref{eq:test-problem}.

Hence the linearized range is exactly the subspace of $L^2(0,1)$
consisting of functions with zero mean. Stationarity is therefore
equivalent to the residual being a constant $c$. Integrating gives
\[
c=f(u_r)-f(u_\ell)-\int_0^1g(x)\,dx,
\]
which also proves the stated condition for the residual to vanish.
\end{proof}

Both arguments require only that the residual function lies in $L^2(0,1)$.
In the inviscid case, a trial function with constant residual is
stationary even when $f'(\bar u)$ vanishes. The nonsonic assumption
is used to show that stationarity forces the residual to be constant.

The nonsonic assumption can be relaxed if the residual has additional
regularity. For {$r=F^0(\cdot;\bar u)\in H^1(0,1)$}, stationarity gives
$f'(\bar u)r_x=0$. If $f'(\bar u)$ is nonzero almost everywhere, then
$r_x=0$ almost everywhere and $r$ is constant.
This argument does not cover the transonic limit studied below,
whose residual has a jump and lies outside $H^1(0,1)$.

For positive viscosity, the residual also controls the trial functions
and their convergence to a solution. We establish this for the
source-free problem $g=0$ with $u_\ell>u_r$.

\begin{proposition}[Convergence to the viscous solution]
\label{prop:viscous-confinement}
Let $\varepsilon>0$, $f\in C^2(\mathbb R)$, $g\equiv0$, and
$u_\ell>u_r$. 
Denote by $u^*_\varepsilon\in H^2(0,1)$ the unique
solution of $F^\varepsilon(\cdot;u^*_\varepsilon)=0$ 
satisfying the boundary conditions \eqref{def:u_l-u_r}.

Every minimizing sequence for $J^\varepsilon$ over
$H^2(0,1)$ functions with these boundary values converges
strongly to $u^*_\varepsilon$ in $H^2(0,1)$.
\end{proposition}

\begin{proof}
Integrating the equation once gives $f(u)-\varepsilon u'=\lambda$.
If $u'$ vanished at a point, uniqueness for the autonomous equation
would imply that $u$ is constant. Thus $u'$ has a fixed nonzero sign,
which is negative because $u_\ell>u_r$. So $\lambda -f(u)>0$. 
Separation of variables gives
\[
\varepsilon\int_{u_r}^{u_\ell}\frac{dz}{\lambda-f(z)}=1,
\qquad \lambda>\max_{[u_r,u_\ell]}f.
\]
Since the left-hand side is continuous and strictly decreasing
in $\lambda$, from $+\infty$ to $0$, there is a unique $\lambda$
satisfying this equation and hence a unique solution
$u^*_\varepsilon$. In particular,
$\|u^*_\varepsilon\|_{L^\infty}\le B:=\max(|u_\ell|,|u_r|)$.

Since $J^\varepsilon(u^*_\varepsilon)=0$, every minimizing sequence
$\{u_n\}$ satisfies $r_n:=F^\varepsilon(\cdot;u_n)\to0$ in $L^2$.
To bound $u_n$, set $w_n=(u_n-B)_+$. Since the boundary
values of $u_n$ are at most $B$, we have $w_n\in H_0^1(0,1)$.
Multiply the residual equation
\[
f'(u_n)u_n'-\varepsilon u_n''=r_n
\]
by $w_n$ and integrate over $(0,1)$. The convection term
integrates to zero. Integration by parts in the diffusion term,
using $u_n'w_n'=|w_n'|^2$ almost everywhere, gives
\[
\varepsilon\|w_n'\|_{L^2}^2
=\int_0^1 r_nw_n\,dx
\le \|r_n\|_{L^2}\|w_n'\|_{L^2}.
\]
Together with $\|w_n\|_{L^\infty}\le\|w_n'\|_{L^2}$, this yields
$u_n\le B+\varepsilon^{-1}\|r_n\|_{L^2}$.
Multiplying instead by $-(-B-u_n)_+$ gives the corresponding
lower bound. Hence
\begin{equation}\label{eq:viscous-residual-linfty}
\|u_n\|_{L^\infty}
\le B+\varepsilon^{-1}\|r_n\|_{L^2}.
\end{equation}

Thus the trial functions are uniformly bounded. The mean-value formula
gives $f(u_n)-f(u^*_\varepsilon)=a_n e_n$, where
$e_n:=u_n-u^*_\varepsilon$ and $\|a_n\|_{L^\infty}$ is bounded
independently of $n$. Subtracting the residual equations yields
\[
(a_n e_n-\varepsilon e_n')'=r_n,\qquad e_n(0)=e_n(1)=0.
\]
The integrating-factor calculation in
Proposition~\ref{prop:viscous-inviscid}(i) gives
\begin{equation}\label{eq:viscous-stability}
\|e_n\|_{L^\infty}+\|e_n'\|_{L^2}
\le C\|r_n\|_{L^2}\longrightarrow0,
\end{equation}
with $C$ independent of $n$. Finally,
\[
\varepsilon e_n''
=f'(u_n)e_n'
+\bigl(f'(u_n)-f'(u^*_\varepsilon)\bigr)
 (u^*_\varepsilon)'-r_n.
\]
Uniform convergence of $u_n$ and continuity of $f'$ show that
the right-hand side tends to zero in $L^2$.
Thus $u_n\to u^*_\varepsilon$ strongly in $H^2(0,1)$.
\end{proof}

These estimates control the trial functions for fixed $\varepsilon>0$,
with constants that deteriorate as $\varepsilon\downarrow0$.
They do not imply bounded neural-network parameters.

\subsection{{Residual minimization and stationary entropy solutions}}
\label{subsec:transonic-escape}

A central question (Q3) is \emph{whether
residual minimization recovers the entropy solution, including when
that solution contains a shock.} Approximating a discontinuity by
smooth trial functions requires increasingly sharp transitions.
For a fixed neural network architecture whose bounded parameter
sets give bounded $H^1$ seminorms, convergence to a shock in $L^1$
therefore forces the parameters to become unbounded.
Parameter escape can thus accompany successful shock approximation.
The more consequential question is what function the approximations
approach.

{Recovery of an entropy solution requires a compatible boundary
formulation: the exact endpoint constraints imposed on the trial
functions differ from entropy boundary conditions, which do not always
enforce both data as trace equalities~\cite{bardos1979first}.
The example below shows failure to recover a stationary shock
even when its traces agree with the prescribed endpoint values.}

{In this subsection, we consider the inviscid problem
\eqref{def:F_eps} with $\varepsilon=0$ and $g\equiv0$.} {The trial functions are the $H^1(0,1)$ functions with the boundary values \eqref{def:u_l-u_r}, that is, the affine space $u_b+X$ with $X=H^1_0(0,1)$, denoted by $\mathcal A$ below.}

Let $f\in C^2(\mathbb R)$ satisfy
\begin{equation}
    \label{assumptions f conservation laws}
    f''>0 \quad \text{and} \quad f(z)\to\infty \quad \text{as $|z|\to\infty$}
\end{equation}
 so it has a unique minimizer $u_{\mathrm s}$, characterized by
$f'(u_{\mathrm s})=0$. We call $u_{\mathrm s}$ the sonic state
and normalize $f(u_{\mathrm s})=0$, which does not change
$\partial_x f(u)$. We first consider \emph{compressive transonic} boundary data:
\[
u_r<u_{\mathrm s}<u_\ell.
\]

{For every $u\in H^1(0,1)$, the chain rule gives
$\partial_x f(u)=f'(u)u_x\in L^2(0,1)$, so that $\mathcal R^0$ maps $u_b+X$ into $L^2(0,1)$ and the functional
$J^0$ defined in \eqref{def:J_eps} is finite.}

The result below gives a negative answer for the strong residual
functional considered here with transonic boundary data.
Its infimum is strictly positive and is not attained in the
admissible class. Every minimizing sequence converges uniformly
to a continuous function outside $H^1$, while its residuals
converge strongly in $L^2$ to a nonzero function.
The limiting profile is therefore not a weak solution.
When the boundary fluxes agree, an entropy-admissible stationary
shock exists, yet residual minimization does not recover it.
Under the parameter bound stated below, the accompanying loss
of regularity forces parameter escape.

\begin{theorem}[Loss of regularity for transonic data]
\label{thm:burgers-cusp}
Let $J^0$ be the functional defined in \eqref{def:J_eps}
with $g\equiv0$, and let $f$ satisfy the assumptions in \eqref{assumptions f conservation laws}.
For transonic boundary data $u_r<u_{\mathrm s}<u_\ell$, set
\[
\mathcal A:=\{u\in H^1(0,1):u(0)=u_\ell,\ u(1)=u_r\},
\]
\[
s:=f(u_\ell)+f(u_r),
\qquad
x^\dagger:=\frac{f(u_\ell)}{s}.
\]
Then:
\begin{enumerate}
\item (Infimum and minimizing sequences) The infimum of $J^0$ over
$\mathcal A$ is
\[
\inf_{u\in\mathcal A}J^0(u)=\tfrac12s^2.
\]
Every minimizing sequence $\{u_n\}\subset\mathcal A$ converges
uniformly on $[0,1]$ to
\[
u^\dagger(x):=
\begin{cases}
f_+^{-1}\!\bigl(s|x-x^\dagger|\bigr), & x\le x^\dagger,\\[2pt]
f_-^{-1}\!\bigl(s|x-x^\dagger|\bigr), & x>x^\dagger,
\end{cases}
\]
where $f_+^{-1}$ and $f_-^{-1}$ are the inverse branches of $f$ on
$[u_{\mathrm s},\infty)$ and $(-\infty,u_{\mathrm s}]$, respectively.
Moreover,
\[
\partial_x f(u_n)\longrightarrow
r^\dagger:=s\operatorname{sign}(x-x^\dagger)
\qquad\text{strongly in }L^2(0,1),
\]
with
\[
u^\dagger\in C^0([0,1])\setminus H^1(0,1),
\qquad
r^\dagger\in L^2(0,1)\setminus H^1(0,1),
\]
and $\|u_n'\|_{L^2(0,1)}\to\infty$.
In particular, the infimum is not attained in $\mathcal A$.

\item (Parameter escape) Let $\mathcal U=\{u_\theta:\theta\in\mathbb R^N\}\subset\mathcal A$
be a parameterized family as in \eqref{space of parametrized functions} such that
\[
\sup_{|\theta|\le R}\|u_\theta'\|_{L^2(0,1)}<\infty
\qquad\text{for every }R>0.
\]
If $J^0(u_{\theta_n})\to\inf_{u\in\mathcal A}J^0(u)$, then
$|\theta_n|\to\infty$.
\end{enumerate}
\end{theorem}

\begin{proof}
\textup{(i)} Write $A:=f(u_\ell)>0$ and $B:=f(u_r)>0$, so $s=A+B$.
For any $u\in\mathcal A$, the flux $w=f(u)\in H^1(0,1)$ satisfies
$w(0)=A$ and $w(1)=B$. Continuity and the transonic boundary data
also give a point $a\in(0,1)$ with $u(a)=u_{\mathrm s}$, hence $w(a)=0$.

For fixed $a$, the least Dirichlet energy among $H^1$ functions
with these three prescribed values is attained by the piecewise
affine interpolant $q_a$ through $(0,A)$, $(a,0)$, and $(1,B)$.
Indeed, $w-q_a$ vanishes at $0,a,1$, while $q_a'$ is constant on
each subinterval. The cross terms therefore vanish, giving
\[
2J^0(u)=\|w'\|_{L^2}^2
=\frac{A^2}{a}+\frac{B^2}{1-a}
 +\|(w-q_a)'\|_{L^2}^2.
\]
Since $x^\dagger=A/s$, elementary algebra yields
\begin{equation}\label{eq:cusp-lower-bound}
2J^0(u)-s^2
=\|(w-q_a)'\|_{L^2}^2
 +\frac{s^2(a-x^\dagger)^2}{a(1-a)}.
\end{equation}
Thus $J^0(u)\ge\tfrac12s^2$ for every $u\in\mathcal A$.

Set $w^\dagger:=q_{x^\dagger}=s|x-x^\dagger|=f(u^\dagger)$.
To approach the lower bound, choose
$0<\delta<\min(x^\dagger,1-x^\dagger)$ and replace $u^\dagger$ on
$[x^\dagger-\delta,x^\dagger+\delta]$ by the straight line joining
its endpoint values. The resulting function $u_\delta$ belongs to
$\mathcal A$: for each fixed $\delta$, its pieces join continuously
and have square-integrable derivatives.
On the modified interval, Taylor expansion at $u_{\mathrm s}$ gives
\[
|f'(u_\delta)|=O(\delta^{1/2}),
\qquad
|u_\delta'|=O(\delta^{-1/2}).
\]
Hence $\partial_x f(u_\delta)$ is uniformly bounded there, while
its magnitude equals $s$ elsewhere. Consequently,
$J^0(u_\delta)=\tfrac12s^2+O(\delta)$, and
\[
\tfrac12s^2
\le\inf_{u\in\mathcal A}J^0(u)
\le\lim_{\delta\to0}J^0(u_\delta)
=\tfrac12s^2.
\]
This proves the asserted value of the infimum.

Now let $\{u_n\}\subset\mathcal A$ be any minimizing sequence,
set $w_n=f(u_n)$, and choose $a_n\in(0,1)$ with
$u_n(a_n)=u_{\mathrm s}$. Identity~\eqref{eq:cusp-lower-bound} gives
\[
a_n\to x^\dagger,
\qquad
\|w_n'-q_{a_n}'\|_{L^2}\to0.
\]
Since $q_{a_n}'\to(w^\dagger)'$ in $L^2$ and
$w_n(0)=w^\dagger(0)$, we obtain $w_n\to w^\dagger$ in $H^1$ and
uniformly. This also proves the asserted residual convergence.
The limiting residual has a nonzero jump at $x^\dagger$ and therefore
does not belong to $H^1$.

Outside any fixed neighborhood of $x^\dagger$, the limiting flux
is bounded away from zero. For large $n$, continuity and the boundary
values therefore force $u_n$ to use the same inverse branches as
$u^\dagger$, giving uniform convergence there.
To control a neighborhood of $x^\dagger$, fix $\varepsilon>0$.
Since both inverse branches are continuous at zero and take the value
$u_{\mathrm s}$ there, choose $\eta>0$ such that
\[
|f_\pm^{-1}(y)-u_{\mathrm s}|<\varepsilon/2
\qquad\text{for }0\le y\le\eta.
\]
Choose $0<\rho<\min(x^\dagger,1-x^\dagger)$ with $s\rho<\eta/2$.
For all sufficiently large $n$, we have
$\|w_n-w^\dagger\|_{L^\infty}<\eta/2$; hence, whenever
$|x-x^\dagger|\le\rho$,
\[
0\le w_n(x)\le s\rho+\|w_n-w^\dagger\|_{L^\infty}<\eta,
\qquad
0\le w^\dagger(x)\le s\rho<\eta.
\]
Both $u_n(x)$ and $u^\dagger(x)$ therefore lie within
$\varepsilon/2$ of $u_{\mathrm s}$, regardless of which inverse branch
$u_n$ uses in this neighborhood. The triangle inequality gives
\[
\sup_{|x-x^\dagger|\le\rho}|u_n(x)-u^\dagger(x)|<\varepsilon.
\]
Together with the uniform convergence on
$[0,x^\dagger-\rho]\cup[x^\dagger+\rho,1]$ established above,
this proves uniform convergence on $[0,1]$.

Taylor expansion at $u_{\mathrm s}$ gives
\[
|(u^\dagger)'(x)|
\sim\sqrt{\frac{s}{2f''(u_{\mathrm s})}}\,
|x-x^\dagger|^{-1/2}
\qquad\text{as }x\to x^\dagger.
\]
Thus $u^\dagger\notin H^1(0,1)$.
If a subsequence of $\{u_n\}$ had bounded derivative norms, the
fixed boundary values would make it bounded in $H^1$.
A further subsequence would then converge weakly in $H^1$;
its uniform limit identifies the weak limit as $u^\dagger$,
a contradiction. Hence $\|u_n'\|_{L^2}\to\infty$.
Finally, an admissible minimizer would generate a constant minimizing
sequence, whose uniform convergence to $u^\dagger\notin\mathcal A$
is impossible. Therefore the infimum is not attained.

\textup{(ii)} If $\{\theta_n\}$ had a bounded subsequence,
the assumed bound would give a corresponding subsequence with
bounded derivative norms, contradicting part~\textup{(i)}.
Hence $|\theta_n|\to\infty$.
\end{proof}

Reflection preserves the residual functional, giving the following
result for reversed transonic data.

\begin{proposition}[Residual minimization for reversed transonic data]
\label{prop:reverse-transonic}
Under the flux assumptions of Theorem~\ref{thm:burgers-cusp},
let $g\equiv0$ and minimize $J^0$ over $H^1(0,1)$ functions with
endpoint values $u_\ell<u_{\mathrm s}<u_r$. Then:
\begin{enumerate}
\item (Minimizing sequences)  Every minimizing sequence converges uniformly on $[0,1]$
to the same strictly increasing continuous function with an
interior cusp. This limit lies outside $H^1(0,1)$, and the parameter-escape conclusion of
Theorem~\ref{thm:burgers-cusp}\textup{(ii)} holds under the
same bounded-parameter assumption.

\item (Entropy comparison) The limit is not a stationary weak
solution. 
\end{enumerate}
\end{proposition}

\begin{proof}
\textup{(i)} The reflection $v(x)=u(1-x)$ exchanges the endpoint
values and preserves both $J^0$ and the $H^1$ seminorm. Applying
Theorem~\ref{thm:burgers-cusp} and reflecting its limiting
profile gives all the assertions in part~\textup{(i)}.

\textup{(ii)} Similarly, the limiting function is strictly increasing and the flux is therefore non-constant throughout the interval. Therefore, this function cannot be a stationary weak solution. 
\end{proof}



\begin{example}[Burgers equation with transonic boundary data]
\label{ex:burgers-transonic}
For the Burgers flux $f(u)=u^2/2$ and boundary data $u(0)=1$,
$u(1)=-1$, Theorem~\ref{thm:burgers-cusp} gives
$x^\dagger=\tfrac12$ and $\inf J^0=\tfrac12$.
Every minimizing sequence converges uniformly to
\[
u^\dagger(x)=\operatorname{sign}\bigl(\tfrac12-x\bigr)
\sqrt{|1-2x|},
\]
whose residual equals $-1$ on $(0,\tfrac12)$ and $+1$ on
$(\tfrac12,1)$.

To apply the result to neural networks, consider the trial family
\begin{equation}\label{eq:hardwired-ansatz}
u_\theta(x)=(1-2x)+x(1-x)\Psi_\theta(x),
\qquad
\Psi_\theta=\sum_{j=1}^m c_j\sigma(a_jx+b_j),
\quad
\sigma=\tanh,
\end{equation}
which satisfies the boundary conditions exactly. Since
$|\sigma|\le1$ and $|\sigma'|\le1$,
\[
\|u_\theta'\|_{L^2}\le\|u_\theta'\|_{L^\infty}
\le 2+\sum_j|c_j|+\tfrac14\sum_j|c_j||a_j|.
\]
Let $\mathcal N=\bigcup_{m\ge1}\mathcal U_m$ be the union of these families over all finite widths, where $\mathcal U_m$ is the family \eqref{eq:hardwired-ansatz} of width $m$, with $N=3m$ parameters.
Every smooth admissible function $v$ can be written as
\[
v(x)=1-2x+x(1-x)q(x),\qquad q\in C^\infty([0,1]).
\]
The $C^1$ density of finite sums of $\tanh$
units~\cite{pinkus1999approximation} therefore gives $C^1$
approximation of $v$ by functions in $\mathcal N$.
Smooth approximation preserving the boundary values and continuity
of $J^0$ in $H^1(0,1)$ then yield
\[
\inf_{u\in\mathcal N}J^0(u)=\tfrac12.
\]
The infimum is not attained, by Theorem~\ref{thm:burgers-cusp}(i).

The derivative bound is independent of the width. For any sequence
$u_{\theta_n}\in\mathcal N$ of arbitrary finite widths $m_n$ with
$J^0(u_{\theta_n})\to\tfrac12$, Theorem~\ref{thm:burgers-cusp}(i)
gives $\|u_{\theta_n}'\|_{L^2}\to\infty$, and hence
\[
\sum_{j=1}^{m_n}|c_{j,n}|\bigl(1+|a_{j,n}|\bigr)
\ge\|u_{\theta_n}'\|_{L^2}-2\longrightarrow\infty.
\]

\smallskip\noindent\textit{Numerical results.}
Figure~\ref{fig:burgers-escape} illustrates the minimizing-sequence behavior
of Theorem~\ref{thm:burgers-cusp} using the
ansatz \eqref{eq:hardwired-ansatz}, with $m=30$ and the hardwired boundary conditions. We use full-batch Adam with learning
rate $10^{-3}$, parameters $(0.9,0.999,10^{-8})$, 
and $2\times10^5$ iterations. Training minimizes a $401$-point
composite-trapezoidal approximation of $J^0$. The viscous comparison
minimizes $J^\varepsilon$ with $\varepsilon=0.05$, using the same
architecture and optimization parameters.

\begin{figure}[htbp]
\centering
\begin{tabular}{cc}
\includegraphics[width=0.42\textwidth]{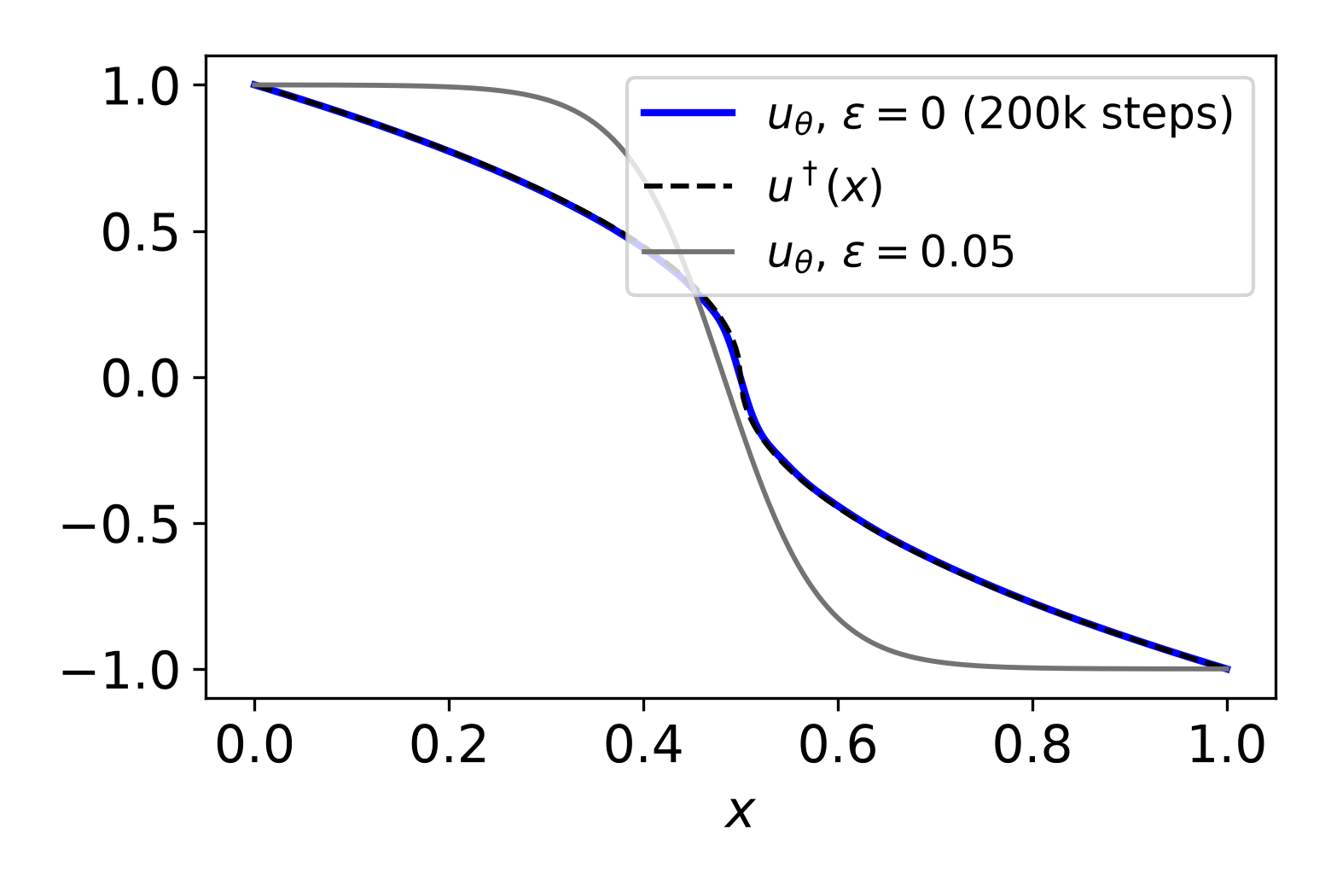}&
\includegraphics[width=0.42\textwidth]{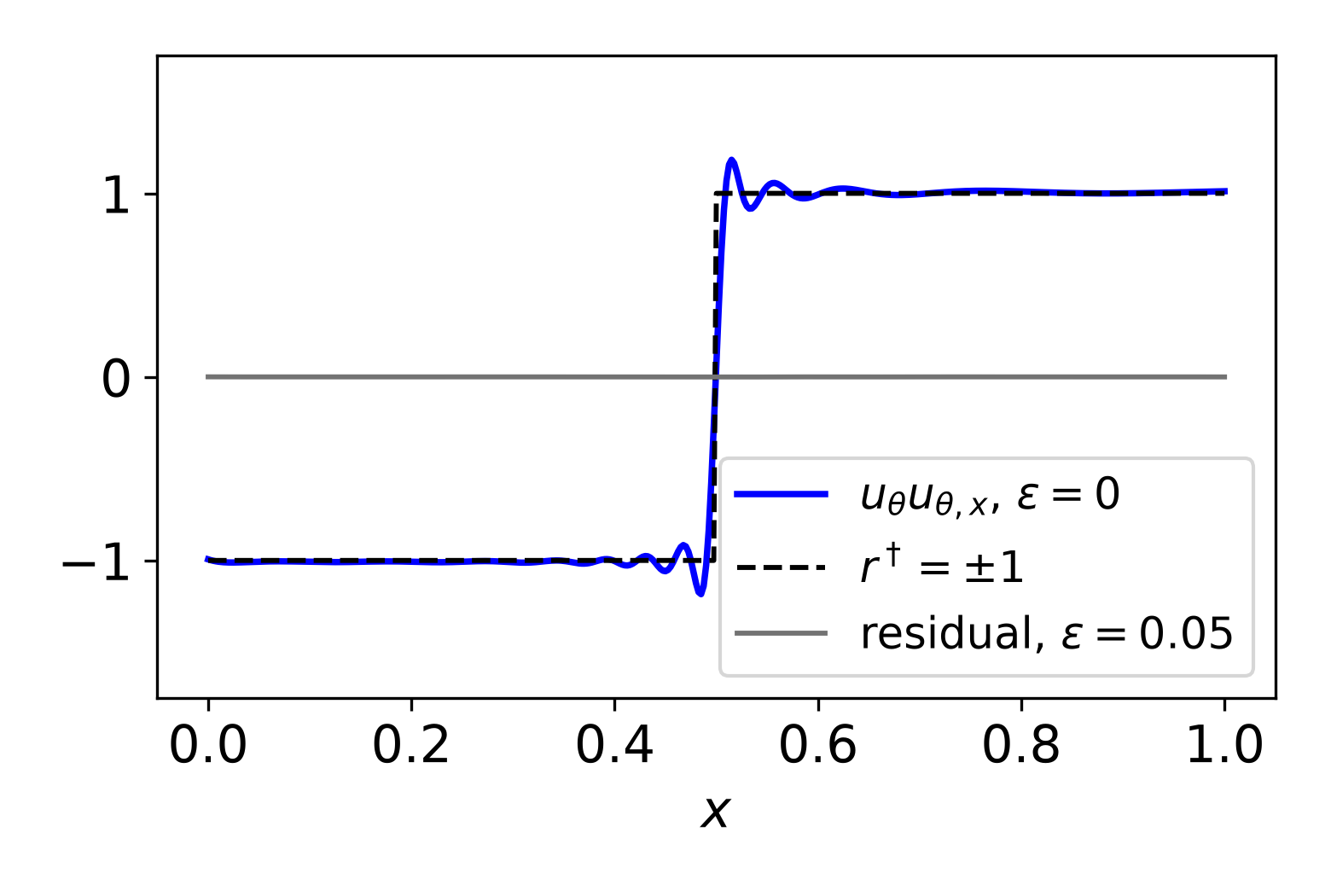}\\
\includegraphics[width=0.42\textwidth]{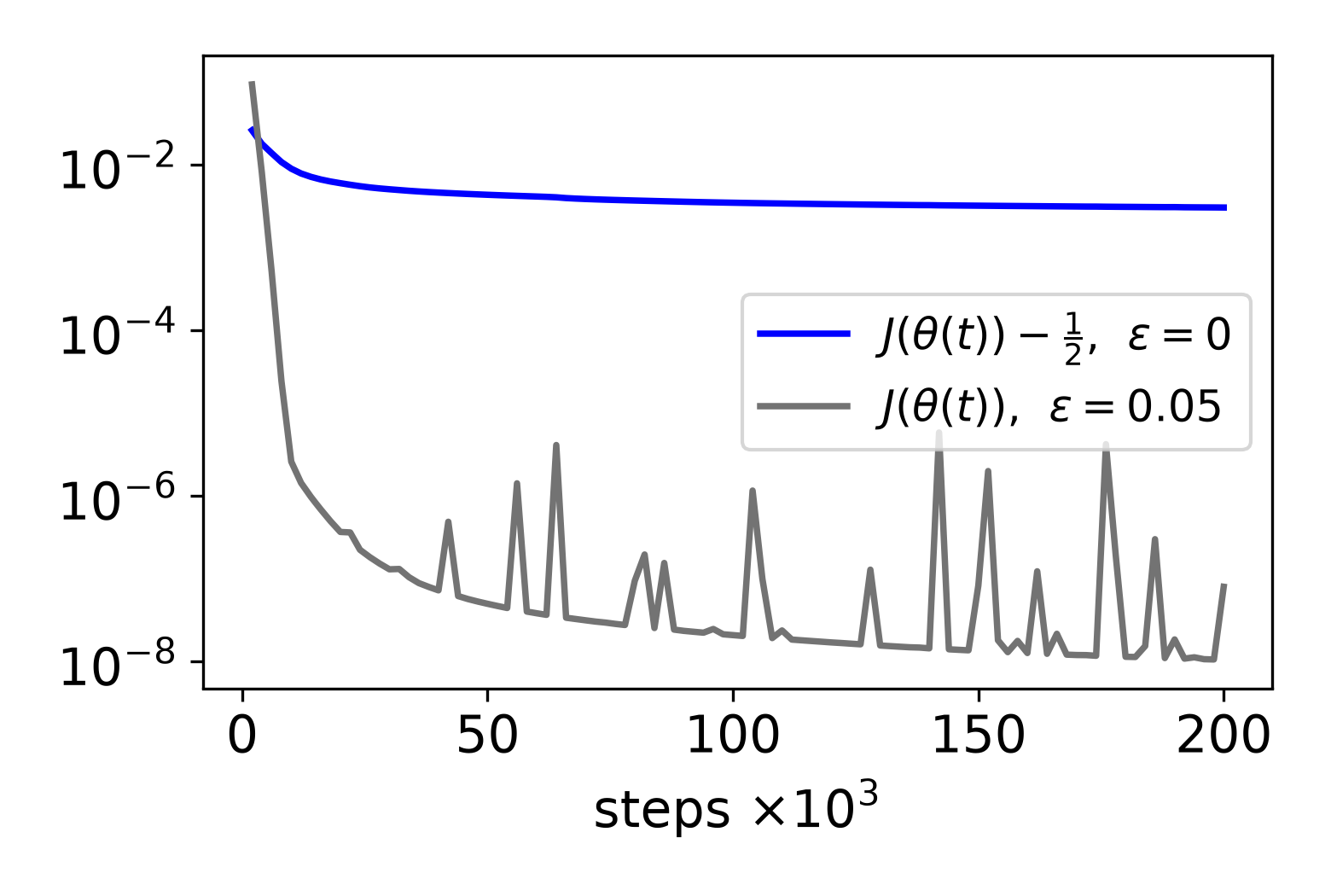}&
\includegraphics[width=0.42\textwidth]{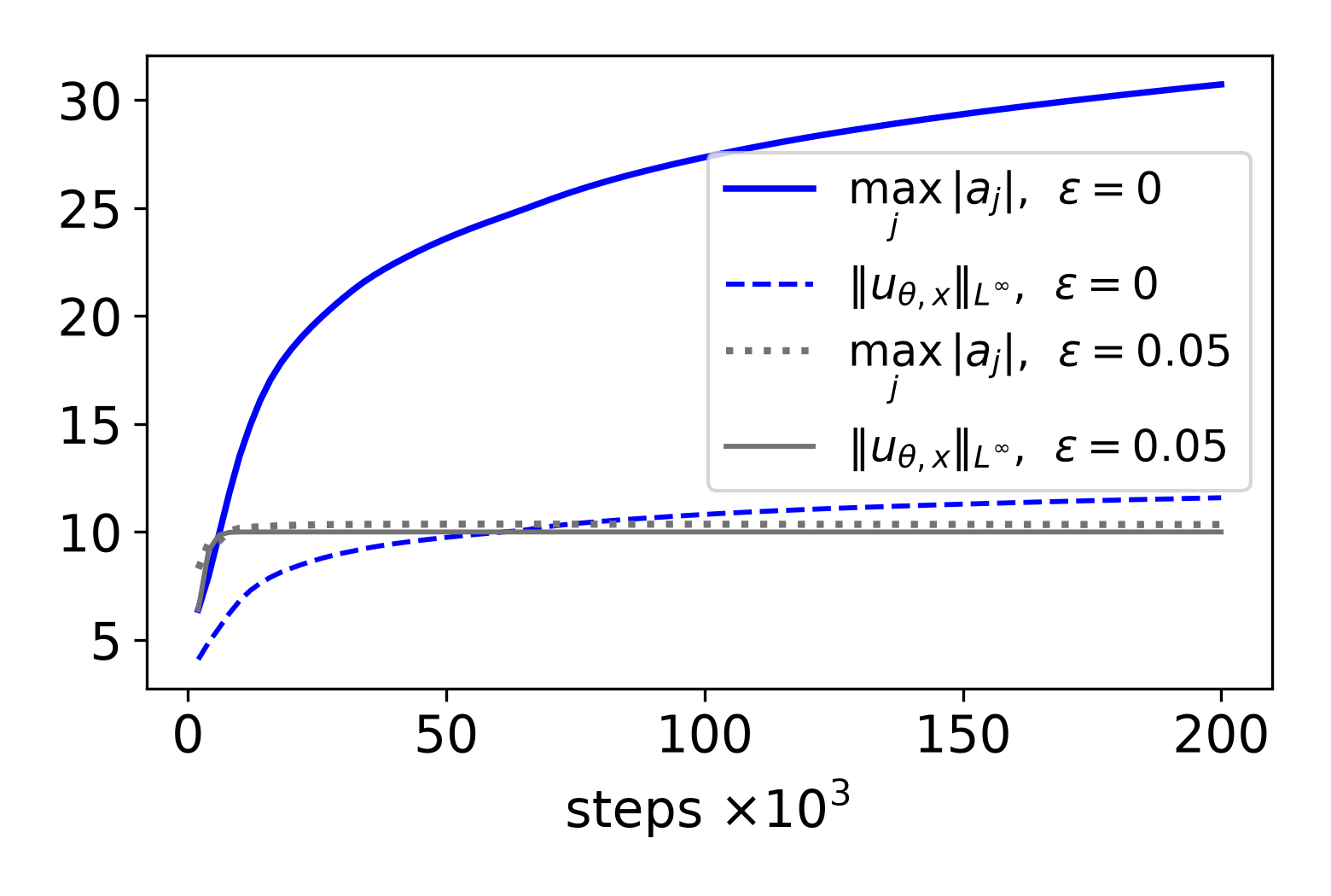}
\end{tabular}
\caption{Training a width-$30$ $\tanh$ network for the inviscid Burgers residual (blue) and the equation with viscosity $\varepsilon=0.05$ (gray). Top left: the trained functions and the predicted inviscid limit $u^\dagger$. Top right: the residuals and the predicted square wave. Bottom left: loss histories, with $1/2$ subtracted in the inviscid case. Bottom right: the inner-weight maxima and derivative maxima, showing continued growth in the inviscid run and stabilization in the viscous run.}
\label{fig:burgers-escape}
\end{figure}

In the inviscid run, the loss decreases toward the predicted infimum
$\tfrac12$, reaching approximately $0.503$, while the trial function
approaches $u^\dagger$ and the residual develops the predicted square-wave
profile with a narrow transition near $x=\tfrac12$. The maximum error
relative to $u^\dagger$ on a refined grid is about $0.05$. Refining the
evaluation grid from $401$ to $102{,}401$ points changes the inviscid loss
by less than $10^{-6}$, much less than its remaining gap of about $0.003$
above the infimum. The slow decay of the approximation error is consistent with the limited regularity of $u^{\dagger}$. 
The viscous run reaches a loss of approximately
$2.44\times10^{-8}$, and its inner weights and derivative maxima stabilize.

Two features deserve emphasis. First, the escape is slow: the observed
inner-weight growth is roughly logarithmic in the iteration count.
In practice, this appears as a plateau at a strictly positive loss while
the parameters continue to grow. Second, no singularity appears in
$u_\theta$ at any finite iteration: every iterate remains analytic, and
the computed profile approximates $u^\dagger$ closely. The loss controls
the derivative of the flux $f(u_\theta)$, while the change in regularity
appears in derivative norms of $u_\theta$ that the loss does not control.

For small viscosity, a narrow shock layer can concentrate contributions
to the gradient and increase the variance under uniform sampling.
Reliable updates may then require larger batches or appropriately
weighted sampling concentrated near the layer.
\end{example}

\Needspace{6\baselineskip}
\section{{Fully nonlinear equations}}\label{sec:fully-nonlinear}

The preceding section asks whether minimizing sequences remain controlled in function space. We now examine the role of the solution concept in interpreting the residual of two fully nonlinear equations. 

{For Monge--Amp\`ere, uniform convexity makes the linearized Dirichlet test problem elliptic and eliminates spurious function-space critical points. The convexity constraint also selects the intended branch of the equation: in two dimensions the interior expression alone cannot distinguish $u$ from $-u$, since $\det D^2(-u)=\det D^2u$. Boundary data may distinguish these functions, but convexity remains part of the elliptic solution class. 

For Hamilton--Jacobi equations, even a vanishing residual can leave multiple pointwise almost-everywhere solutions. The eikonal example shows explicitly how the viscosity condition imposes information at kinks that is absent from the $L^2$ residual.}

\subsection{The Monge--Amp\`ere equation}
\label{subsec:monge-ampere}

 The Monge--Amp\`ere equation is fully nonlinear in the Hessian. In two dimensions, its operator also has the divergence representation $\det D^2u=\tfrac12\nabla\cdot(\operatorname{cof}(D^2u)\nabla u)$ for smooth $u$. Its linearization at a uniformly convex function is a uniformly elliptic operator in divergence form. The corresponding Dirichlet problem therefore provides a two-dimensional application of the dense-range criterion.

Consider the Dirichlet problem for the elliptic Monge--Amp\`ere equation on
\[
\Omega:=(0,1)^2:
\qquad
\det D^2u = f(x,y)
\quad \text{in }\Omega,
\qquad
u=\varphi
\quad \text{on }\partial\Omega,
\]
where $f\in L^2(\Omega)$ and $f>0$ almost everywhere,
together with the convexity constraint
\[
u \text{ is convex in } \Omega.
\]
We consider the least-squares functional
\[
J(u)
:=
\frac12\int_\Omega \big(\det D^2u-f\big)^2\,dx,
\]
{that is, the loss functional of Section~\ref{subsec:tiers} for the residual map $\mathcal R(u)=\det D^2u-f$, with $Y=L^2(\Omega)$, with $X=\{v\in C^2(\overline\Omega):v=0\text{ on }\partial\Omega\}$, and with trial functions in $u_b+X$ for a fixed $u_b\in C^2(\overline\Omega)$ of trace $\varphi$. Then $\mathcal R\in C^1(X;Y)$, with $D\mathcal R(u)[v]=DF(u)[v]=\operatorname{cof}(D^2u):D^2v$ for $v\in X$.}

Least-squares and penalty formulations of the Monge--Amp\`ere equation, including the necessity of strictly enforcing convexity, have been developed in the numerical literature \cite{brenner2011c0,caboussat2013least,prins2015least,brenner2024nonlinear}.

{Input-convex neural networks impose convexity of $x\mapsto u_\theta(x)$
through architectural constraints~\cite{amos2017input}.
Nystr{\"o}m and Vestberg~\cite{nystrom2023dirichlet} use smooth input-convex
networks for the Dirichlet Monge--Amp\`ere problem. However, architectural
convexity alone does not supply the positive lower Hessian bound
assumed below or guarantee convexity of the loss in $\theta$.}

Let \(\bar u\in C^3(\overline\Omega)\) be a critical point of
\(J\) with the prescribed boundary trace \(\varphi\), and define
\[
r:=\det D^2\bar u-f\in L^2(\Omega).
\]
Assume moreover that \(\bar u\) is uniformly convex. Since
\(\bar u\in C^3(\overline\Omega)\) and \(\overline\Omega\) is compact, there
are constants \(0<\lambda\le\Lambda<\infty\) such that
\[
\lambda I\le D^2\bar u(x)\le \Lambda I
\qquad \text{for all }x\in \overline{\Omega}.
\]
Then sufficiently small \(C^2\) perturbations preserve convexity, so the
first variation may be tested against arbitrary smooth variations \(v\)
with \(v=0\) on \(\partial\Omega\); since
\(v\mapsto DJ(\bar u)[v]\) is continuous in the \(H^2\)
norm and smooth Dirichlet variations are dense, stationarity extends to
every \(v\in H^2(\Omega)\cap H_0^1(\Omega)\).
{This extends the first-variation identity, rather than the nonlinear residual map, to $H^2$ variations.}

Set $A:=\operatorname{cof}(D^2\bar u)$. The matrix $A$ is symmetric,
and a direct calculation shows that its rows are divergence free.
Thus
\[
DF(\bar u)[v]
=A:D^2v
=\nabla\cdot(A\nabla v).
\]
In two dimensions, the eigenvalues of $A$ are those of $D^2\bar u$
interchanged, so uniform convexity gives
\[
\lambda I\le A\le\Lambda I
\qquad\text{on }\overline\Omega.
\]
Moreover, $\bar u\in C^3(\overline\Omega)$ ensures that
$A\in C^1(\overline\Omega)$.

For arbitrary $\tilde g\in L^2(\Omega)$, consider the Dirichlet test
problem
\[
\nabla\cdot(A\nabla v)=\tilde g
\quad\text{in }\Omega,
\qquad
v=0\quad\text{on }\partial\Omega.
\]
Uniform ellipticity gives a unique weak solution $v\in H_0^1(\Omega)$.
Global $H^2$ regularity for this problem on the convex polygon gives
$v\in H^2(\Omega)$; see~\cite{grisvard1985elliptic}. This solution can
therefore be used in the extended stationarity identity, yielding
\[
0
=\int_\Omega r\,\nabla\cdot(A\nabla v)\,dx
=\int_\Omega r\,\tilde g\,dx
\qquad\forall\,\tilde g\in L^2(\Omega).
\]
Taking $\tilde g=r$ gives $r=0$. Thus every uniformly convex critical
point of class $C^3$ solves the Monge--Amp\`ere equation.

The argument uses only $r\in L^2(\Omega)$ and does not differentiate
$r$. Uniform convexity plays here the role that viscosity plays in
Proposition~\ref{prop:viscous-inviscid}(i): it provides the solvability
of the test problem.

Without uniform convexity, the cofactor matrix may become degenerate;
if convexity is also dropped, it may be indefinite. Solvability of the
test problem is then no longer assured.
Spurious critical points can occur even when a smooth uniformly convex
solution exists. For this example, replace the square $(0,1)^2$ by the
unit disk $\Omega=B_1(0)\subset\mathbb R^2$, and take $f\equiv1$ and
homogeneous Dirichlet data. The function $u^*(x)=\tfrac12(|x|^2-1)$ is a uniformly convex
solution, whereas $\bar u=0$ satisfies
\[
r=-1,\qquad \operatorname{cof}(D^2\bar u)=0,\qquad
DJ(0)[v]=0
\]
for every smooth Dirichlet variation $v$, although $J(0)=\pi/2>0$.
Thus $\bar u=0$ is a spurious critical point. It is convex but not
uniformly convex, so convexity alone does not exclude this degeneracy.

For a finite-dimensional parameterization, the conclusion that uniformly convex function-space critical points solve the equation does not transfer automatically. Proposition~\ref{prop:parametric-fredholm} replaces the full image $D\mathcal R(u_{\theta^*})[X]$ by $D\mathcal R(u_{\theta^*})[\mathcal V_{\theta^*}]$; the latter is finite-dimensional for a fixed finite model and therefore cannot be dense in $L^2(\Omega)$. The function-space result should consequently be read as identifying the PDE mechanism that eliminates spurious critical points when all admissible variations are available. Whether a particular parameterized or collocation discretization inherits that property is a separate approximation and rank question, exactly as in Sections~\ref{subsec:parametric-functions} and~\ref{sec:linear-pdes}.

\subsection{Hamilton--Jacobi equations: zero residual without viscosity selection}
\label{subsec:hj-ae}

The preceding analysis shows that critical points can have nonzero
residuals and that minimizing sequences can lose the required regularity.
Hamilton--Jacobi equations raise a further question: even a function with
zero residual may fail to be the intended solution. The equation can hold
almost everywhere without satisfying the viscosity-solution criterion.

Consider
\[
F(x;u)=H(u_x)-g(x),
\qquad
J(u)=\frac12\int_0^1F(x;u)^2\,dx,
\]
over Lipschitz functions with prescribed Dirichlet data and define
\[
S:=\{u\in\operatorname{Lip}[0,1]:H(u_x)=g\ \text{a.e.},\ u(0)=u_0,\ u(1)=u_1\}{\subset u_b+X},
\]
where $X=\{v\in W^{1,\infty}(0,1):v(0)=v(1)=0\}.$
Every $u\in S$ has $F(\cdot;u)=0$ in $L^2$ and hence $J(u)=0$.

For the eikonal problem $|u_x|^2=1$ with $u(0)=u(1)=0$, the set $S$ contains a continuum of sawtooth profiles with slopes $\pm1$, whereas the viscosity solution is the distance function $\min(x,1-x)$ \cite{crandall1992user}. 

The first variation confirms that no hidden stationarity condition repairs
this loss of selection. Assume \(g\in C^1([0,1])\), let \(H\in C^2\) on a
neighborhood of the essential range of \(u_x\), and {suppose that $u$ is $C^2$ up to the closure of each cell $I_j$ of a finite partition containing its kinks.}

Set $r(x):=F(x;u)=H(u_x(x))-g(x)$. Then
\begin{equation}\label{eq:hj-cellwise-variation}
DJ(u)[v]
=
-\sum_j\int_{I_j}\partial_x\bigl(rH'(u_x)\bigr)v\,dx
-
\sum_i\jump{rH'(u_x)}_{x_i}v(x_i),
\end{equation}
 for variations $v\in X$, that is, $v\in W^{1,\infty}(0,1)$ with $v(0)=v(1)=0$. 
 
Stationarity requires $\partial_x(rH'(u_x))=0$ on each cell and
$\jump{rH'(u_x)}_{x_i}=0$ at each interface. Writing
$p_\pm=u_x(x_i^\pm)$, the interface condition reads
\[
\bigl(H(p_-)-g(x_i)\bigr)H'(p_-)
=
\bigl(H(p_+)-g(x_i)\bigr)H'(p_+).
\]
On $S$, $r$ vanishes identically on every cell, so both the
cellwise and interface conditions hold automatically and impose
no ordering on the one-sided slopes. For the eikonal equation
$|u_x|=1$, stationarity therefore permits both a local maximum,
with $(p_-,p_+)=(1,-1)$, and a local minimum, with
$(p_-,p_+)=(-1,1)$.

The viscosity condition distinguishes these two corners.
At a local minimum, the constant test function touching from below
has derivative zero and violates the supersolution inequality
$|\phi_x|-1\ge0$, whereas the downward corner satisfies the
viscosity conditions. Within the piecewise-linear zero-residual
class, semiconcavity excludes upward slope jumps and, together
with $u(0)=u(1)=0$, selects the distance profile
$u(x)=\min(x,1-x)$. This characterization by semiconcavity is
specific to the eikonal example considered here. More generally,
an additional admissibility criterion is needed to identify the
intended solution among zero-residual functions. It may be
enforced through an appropriate constraint, a vanishing-viscosity
approximation, or a consistent monotone discretization with the
stability, comparison, and residual estimates required to justify
convergence to the viscosity solution
\cite{crandall1992user,esteve2025finite,barles1991convergence}.

\section{Minimal residual regularity and consistency of sampled gradients}
\label{sec:piecewise-training}

The minimal spatial regularity needed to define a strong residual
in $L^2$ ensures that the least-squares functional is well defined.
It does not, by itself, ensure that gradients obtained by
differentiating sampled losses consistently approximate the gradient
of the continuous functional. We return to this distinction,
introduced in Section~\ref{subsec:sampled-gradient}, and examine
the role of the trial family's parameter dependence.

For piecewise-smooth parameterized trial families, we first derive
the first variation on a fixed partition and then consider
interfaces that move with the parameters. The cellwise adjoint
calculations concern second-order operators, while the comparison
of sampled and continuous gradients is formulated directly in
terms of the residual and therefore also applies to higher-order
equations under the stated hypotheses.

\subsection{Cellwise adjoints and interface conditions}
\label{subsec:piecewise-setup}

Fix a partition
\[
\Pi=\{0=x_0<x_1<\cdots<x_{N_c}=1\},
\]
and write $C^m_\Pi$ for functions whose restriction to each cell
$I_j=(x_{j-1},x_j)$ extends to a $C^m$ function on
$[x_{j-1},x_j]$, with no matching at the interfaces built into this
notation. The interior points $x_1,\ldots,x_{N_c-1}$ are the
interfaces. For a function $h$ with one-sided traces, write
\[
\jump{h}_{x_i}:=h(x_i^+)-h(x_i^-),
\qquad
\overline h_{x_i}:=\frac12\bigl(h(x_i^+)+h(x_i^-)\bigr).
\]
Consider a differential operator of the form \eqref{eq:F2},
\begin{equation}\label{eq:general-second-order-residual}
F(x;u)=\mathcal F(x,u,u_x,u_{xx})-g(x),
\end{equation}
where $\mathcal F=\mathcal F(x,z,p,q)$ is smooth. This includes operators
that depend nonlinearly on the highest derivative $u_{xx}$.
For $r:=F(\cdot;u)$, define the boundary and interface functional
\[
\mathcal S_u(r,v)
:= \sum_j\left[
\bigl\langle r,DF(u)[v]\bigr\rangle_{L^2(I_j)}
-
\bigl\langle DF(u)^*r,v\bigr\rangle_{L^2(I_j)}
\right].
\]

With all coefficients evaluated along $u$, write
\[
a:=\mathcal F_q,\qquad b:=\mathcal F_p,\qquad c:=\mathcal F_z.
\]
The linearization \eqref{eq:formal-linearization} and its cellwise formal adjoint are
\begin{equation}\label{eq:general-cellwise-adjoint}
DF(u)[v]=a v_{xx}+b v_x+c v,
\qquad
DF(u)^*r=(ar)_{xx}-(br)_x+cr.
\end{equation}
Then, through integration by parts, 
\[
\mathcal S_u(r,v)
= 
\sum_{j=1}^{N_c}
\Bigl[ar\,v_x+\bigl(br-(ar)_x\bigr)v\Bigr]_{x_{j-1}^+}^{x_j^-},
\]
containing only information across the cell interfaces. 
The following identity separates the cellwise adjoint equation from
the conditions on these traces.

\begin{proposition}[Cellwise adjoint identity and criticality]
\label{prop:piecewise-adjoint}
Let $F$ be given by \eqref{eq:general-second-order-residual}, and assume
\[
u\in C^1([0,1])\cap C^4_\Pi,
\qquad
g\in C^2_\Pi.
\]
Set $r:=F(\cdot;u)$ and $I_j:=(x_{j-1},x_j)$. Let $\mathcal V_\Pi$
be the space of all $v\in C^1([0,1])\cap C^2_\Pi$
satisfying the prescribed homogeneous endpoint conditions.
Then, for every $v\in\mathcal V_\Pi$,
\begin{equation}\label{eq:piecewise-adjoint-identity}
DJ(u)[v]
=
\sum_{j=1}^{N_c}\int_{I_j}(DF(u)^*r)\,v\,dx
+\mathcal S_u(r,v).
\end{equation}
Moreover, $u$ is a critical point of $J$ with respect to
$\mathcal V_\Pi$ if and only if both the cellwise adjoint equations
\begin{equation}\label{eq:cellwise-adjoint}
DF(u)^*r=0
\qquad\text{on each }I_j
\end{equation}
and the boundary and interface condition
\[
\mathcal S_u(r,v)=0
\qquad\text{for every }v\in\mathcal V_\Pi
\]
hold.
\end{proposition}

\begin{proof}
The regularity assumptions make $r$, $a$, and $b$ cellwise $C^2$.
Differentiating $J(u+t v)$ at $t=0$ gives the sum of the integrals
$\int r\,DF(u)[v]\,dx$ over the fixed cells. On each cell,
\[
\begin{aligned}
\int r(a v_{xx}+b v_x+c v)\,dx
&=\int\bigl((ar)_{xx}-(br)_x+cr\bigr)v\,dx\\
&\quad+\Bigl[ar\,v_x+\bigl(br-(ar)_x\bigr)v\Bigr].
\end{aligned}
\]
Summing gives \eqref{eq:piecewise-adjoint-identity}. Variations compactly
supported in a single cell yield \eqref{eq:cellwise-adjoint}, and the
remaining assertion follows from the identity.
\end{proof}

If $\mathcal V_\Pi$ consists of all $v\in C^1([0,1])\cap C^2_\Pi$ with
$v(0)=v(1)=0$, the boundary and interface conditions are
\begin{equation}\label{eq:general-adjoint-traces}
\begin{gathered}
(ar)(0)=(ar)(1)=0,\\
\jump{ar}_{x_i}=0,\qquad
\jump{br-(ar)_x}_{x_i}=0,\qquad 1\le i<N_c.
\end{gathered}
\end{equation}

The traces involve the residual multiplied by coefficients of the
linearized operator. For a fully nonlinear equation, these coefficients
can jump because they depend on $u_{xx}$, even when $u$ is globally
$C^1$. Concluding that the residual vanishes requires uniqueness for the
adjoint problem with the associated boundary and interface conditions,
or a separate dense-range argument. Stationarity with respect to
finitely many parameters does not, by itself, impose all these
conditions or the cellwise adjoint equation.

\subsection{{Parameter-dependent interfaces and bias in the sampled gradient}}
\label{subsec:sampled-bias}

We now consider a parameterized trial family $u_\theta$ whose smooth
cells can change as the parameters vary. For neural networks with
piecewise-smooth activations, the interfaces are preimages of the
activation breakpoints under the pre-activation maps. Changing the weights and biases can
therefore move the interface locations $x_i(\theta)$. For brevity, we call such interfaces moving interfaces.

We write the
parameter-dependent partition as
\[
\Pi_\theta=\{0=x_0<x_1(\theta)<\cdots<x_{N_c-1}(\theta)<x_{N_c}=1\}
\]
with $\theta\in\mathbb R^N$. The number $N_c$ of cells is
fixed locally in parameter space, and the interior interfaces
$x_i(\theta)$ are $C^1$. Write
$I_j(\theta):=(x_{j-1}(\theta),x_j(\theta))$, and use the jump and
average conventions of Section~\ref{subsec:piecewise-setup}.
Endpoint values denote traces from inside $[0,1]$.

The sampled gradient differentiates residual evaluations at fixed
spatial points. Differentiating the continuous loss also accounts
for motion of the cell boundaries. We use $J$ and $\hat J$ for the
continuous and sampled losses, and $G$ for the mean sampled gradient,
as in Section~\ref{subsec:sampled-gradient}.

\begin{proposition}[Bias of the sampled gradient]
\label{prop:sgd-bias}
Let $\Pi_\theta$ be the moving partition above. Assume that the
restriction of $r(x,\theta)$ to each moving cell is jointly $C^1$
in $x$ and $\theta$, up to the cell endpoints. The one-sided
traces at an interface may differ. Define
\[
J(\theta)=\frac12\int_0^1 r(x,\theta)^2\,dx.
\]
Then, for $\theta = (\theta_1,\cdots, \theta_N)$, 
the partial derivatives of $J(\theta)$ are
\begin{equation}\label{eq:sgd-bias}
\partial_{\theta_k}J(\theta)
=G_k(\theta)-\frac12\sum_{i=1}^{N_c-1}
\jump{r^2}_{x_i}\,\partial_{\theta_k}x_i,\qquad k=1,\ldots,N
\end{equation}
where
\[
G_k(\theta):=\sum_{j=1}^{N_c}\int_{I_j(\theta)}
r(x,\theta)\,\partial_{\theta_k}r(x,\theta)\,dx,
\]
and $\partial_{\theta_k}r$ is taken at fixed $x$ within each cell.
Set $G(\theta):=(G_1(\theta),\ldots,G_N(\theta))^T$.
In particular, $\nabla_\theta J=G$ if $r$ is continuous across every
interface.
\end{proposition}

\begin{proof}
Differentiating the integral over one moving cell gives
\[
\begin{aligned}
\partial_{\theta_k}\left(\frac12\int_{x_{j-1}(\theta)}^{x_j(\theta)}
r(x,\theta)^2\,dx\right)
&=\int_{I_j(\theta)}r\,\partial_{\theta_k} r\,dx\\
&\quad+\frac12r(x_j^-,\theta)^2\partial_{\theta_k} x_j
-\frac12r(x_{j-1}^+,\theta)^2\partial_{\theta_k} x_{j-1}.
\end{aligned}
\]
The outer endpoints are fixed. Summing over the cells combines the
two contributions at each interior interface into
$-\tfrac12\jump{r^2}_{x_i}\partial_{\theta_k} x_i$, while the integral
terms sum to $G_k$. This proves \eqref{eq:sgd-bias}.
\end{proof}

For $r=F(\cdot;u_\theta)$, write
$w_k:=\partial_{\theta_k}u_\theta$ within each cell. 
$$G_k(\theta)=\sum_{j=1}^{N_c}
\bigl\langle r,DF(u_\theta)[w_k]\bigr\rangle_{L^2(I_j(\theta))}$$ 
corresponds to the cell integrals underlying the fixed-partition first
variation in Proposition~\ref{prop:piecewise-adjoint}. When each $w_k$
belongs to $\mathcal V_{\Pi_\theta}$, that proposition decomposes $G_k$
into the adjoint volume term and $\mathcal S_{u_\theta}(r,w_k)$.
The additional term in \eqref{eq:sgd-bias} accounts for motion of
the cell boundaries. In vector form, write
\begin{equation}\label{eq:interface-correction}
\nabla_\theta J=G+K,
\qquad
K(\theta):=-\frac12\sum_{i=1}^{N_c-1}
\jump{r^2}_{x_i}\,\nabla_\theta x_i.
\end{equation}

For independent uniform samples $\xi_1,\dots,\xi_n$ in $(0,1)$, the sampled loss
\eqref{eq:sampled-loss} satisfies
\[
\mathbb E[\hat J(\theta;\xi)]=J(\theta),
\qquad
\mathbb E[\nabla_\theta\hat J(\theta;\xi)]=G(\theta).
\]
Consequently,
\begin{equation}\label{eq:sampled-gradient-bias-vector}
\mathbb E[\nabla_\theta\hat J(\theta;\xi)]=\nabla_\theta J(\theta)
-K(\theta).
\end{equation}
The loss estimator is unbiased, but its gradient has a bias
independent of the sample size. The bias vanishes precisely when
the interface contributions cancel. In particular, continuity of
the residual is sufficient. For a smooth second-order expression
with continuous forcing, globally $C^2$ trial functions have
continuous residuals and therefore give $K=0$ under the proposition's
parameter-regularity assumptions. The proposition itself uses only the
residual and applies equally to higher-order equations satisfying
its hypotheses.

\begin{example}[Zero sampled gradient with nonzero continuous derivative]
\label{ex:zero-sampled-gradient}
For $0<\theta<1$, let $\sigma(z):=(\max\{z,0\})^2$ be the
squared-ReLU activation and consider the network with one hidden
layer and two units,
\[
\Psi_\theta(x)=-\frac12\sigma(x)-\frac12\sigma(x-\theta).
\]
The weights are fixed, and only the bias $-\theta$ varies. For the
Poisson problem $-u''=1$ on $(0,1)$, enforce the homogeneous
Dirichlet data by setting
\[
u_\theta(x)=\Psi_\theta(x)-(1-x)\Psi_\theta(0)-x\Psi_\theta(1).
\]
Each $u_\theta$ is globally $C^1$, piecewise quadratic, and satisfies
$u_\theta(0)=u_\theta(1)=0$. The affine boundary correction has zero
second derivative, so the strong residual and loss are
\[
r(x,\theta)=-u_\theta''(x)-1=\mathbf1_{\{x>\theta\}}
\quad\text{a.e.},
\qquad J(\theta)=\frac{1-\theta}{2}.
\]
At each fixed $x\ne\theta$, the residual is locally constant in
$\theta$, so $G(\theta)=0$. The single interface is $x_1(\theta)=\theta$,
with $\jump{r^2}_{x_1}=1$ and $x_1'(\theta)=1$. Thus
\[
K(\theta)=-\frac12,
\qquad J'(\theta)=G(\theta)+K(\theta)=-\frac12.
\]
For each fixed $\theta$, a finite uniformly sampled point set almost
surely avoids the interface, giving an exactly zero sampled
derivative with respect to $\theta$. Nevertheless, the sampled loss converges almost surely
to $J(\theta)$ as the sample size increases. This discrepancy is
independent of sampling variance or optimizer behavior. The map
$\theta\mapsto u_\theta$ is not $C^1$ into $X=H^2(0,1)\cap H^1_0(0,1)$, on which
$\mathcal R(u)=-u''-1$ is $C^1$ into $L^2(0,1)$, so the hypotheses of
Section~\ref{subsec:parametric-functions} fail and Proposition~\ref{prop:parametric-fredholm}
and Theorem~\ref{thm:linear-projection} do not apply.
\end{example}

\begin{example}[Equations in divergence form]
\label{ex:divergence-traces}
Let $u_\theta\in C^1([0,1])\cap C^4_{\Pi_\theta}$ satisfy fixed
Dirichlet data. Assume that $u_\theta$ and its first two spatial
derivatives are jointly $C^1$ in $x$ and $\theta$ on each cell,
up to the cell endpoints. Consider
\[
F(x;u)=\partial_x f(x,u,u_x)-g(x),
\]
where $f$ is smooth and $g\in C^2([0,1])$. Evaluate $f_u$ and $f_p$
along $u_\theta$. Then
\[
DF(u_\theta)[w]=\partial_x(f_u w+f_p w_x),
\qquad
DF(u_\theta)^*r=(f_p r_x)_x-f_u r_x.
\]
Collect the residual traces and the corresponding traces of a
parameter variation in
\[
\begin{aligned}
\varrho(r)&:=\Bigl(r(0),\,r(1),\,
(\jump{r_x}_{x_i})_{i=1}^{N_c-1},\,
(\jump{r}_{x_i})_{i=1}^{N_c-1}\Bigr),\\
\tau(w)&:=\Bigl(-(f_p w_x)(0),\,(f_p w_x)(1),\,
(f_p(x_i)w(x_i))_{i=1}^{N_c-1},\\
&\hspace{3em}
\bigl(-f_u(x_i)w(x_i)-f_p(x_i)\overline{w_x}_{x_i}\bigr)_{i=1}^{N_c-1}\Bigr).
\end{aligned}
\]
Both vectors have $2N_c$ components. With
$w_k=\partial_{\theta_k}u_\theta$, Proposition~\ref{prop:sgd-bias}
and cellwise integration by parts give
\begin{equation}\label{eq:defect-trace-pairing}
\nabla_\theta J(\theta)
=\left(
\sum_{j=1}^{N_c}
\bigl\langle DF(u_\theta)^*r,w_k\bigr\rangle_{L^2(I_j)}
+\varrho(r)\cdot\tau(w_k)
\right)_{k=1}^{N}.
\end{equation}
The interface velocities cancel explicitly because
\[
\jump{w_{k,x}}_{x_i}
=-\jump{u_{\theta,xx}}_{x_i}\,\partial_{\theta_k}x_i,
\qquad
\jump{r}_{x_i}=f_p(x_i)\jump{u_{\theta,xx}}_{x_i}.
\]
These identities, together with
$\overline r_{x_i}\jump{r}_{x_i}=\tfrac12\jump{r^2}_{x_i}$,
cancel the interface-motion contribution from integration by
parts against $K_k$. Thus interface motion is represented through
the sensitivities $w_k$ and their traces, with no explicit
$\partial_{\theta_k}x_i$ in \eqref{eq:defect-trace-pairing};
this does not require $K_k=0$.

Geometrically, parameter criticality is precisely orthogonality
of $(DF(u_\theta)^*r,\varrho(r))$ to the span of
$(w_k,\tau(w_k))$, $k=1,\ldots,N$, in
$L^2(0,1)\times\mathbb R^{2N_c}$, with the cellwise $L^2$ and
Euclidean inner products. The volume and trace contributions
remain together in this condition and can cancel within each
parameter direction.

\end{example}

The adjoint calculation uses the second-order structure of the
operator. Higher-order equations require the corresponding
higher-order trace terms; the residual-level identity
$\nabla_\theta J=G+K$ is unchanged.

\subsection{{Numerical comparison for the Poisson equation}}
\label{subsec:poisson-comparison}
We use the Poisson experiment of Figure~\ref{fig:elu-results} to
examine how interface bias and sampling variance affect training.
The comparison varies two choices: how $G$ is evaluated, and
whether the interface correction $K$ from
\eqref{eq:interface-correction} is included. Writing $G_Q$ for a
cellwise quadrature approximation of $G$, Table~\ref{tab:poisson-terminal-losses} lists the four vectors supplied
to the optimizer; $G_Q$ approximates the mean sampled gradient and $G_Q+K$ approximates the continuous gradient.

\begin{table}[!htbp]
\centering
\caption{{The four gradients handed to the optimizer in the Poisson comparison, and the terminal residual losses after $20{,}000$ iterations, evaluated by the same cellwise quadrature for all variants. Each loss entry is the range over six seeds, in units of $10^{-3}$.}}
\label{tab:poisson-terminal-losses}
\begin{tabular}{@{}lcccc@{}}
\toprule
 & \multicolumn{2}{c}{Without $K$} & \multicolumn{2}{c}{With $K$} \\
\cmidrule(lr){2-3}\cmidrule(l){4-5}
Gradient evaluation & Vector & Loss ($\times10^{-3}$) & Vector & Loss ($\times10^{-3}$) \\
\midrule
Cellwise quadrature & $G_Q$ & $20$--$42$ & $G_Q+K$ & $0.83$--$5.0$ \\
Uniform sampling, $100$ points & $\nabla_\theta\hat J$ & $45$--$110$ & $\nabla_\theta\hat J+K$ & $1.9$--$13$ \\
\bottomrule
\end{tabular}
\end{table}

Both $G_Q$ and $K$ use the interface locations recomputed at each
iteration, using explicit formulas for the first layer and
numerical root finding for the second to locate the sign changes of the respective layer inputs. To evaluate $K$, we
compute $\jump{r^2}_{x_i}$ from one-sided residual values at each
interface. The sensitivities $\nabla_\theta x_i$ are obtained by
differentiating the explicit first-layer locations and implicitly
differentiating the second-layer root equations. We then assemble
$K$ using \eqref{eq:interface-correction}. In both corrected
variants, this calculation uses the current network parameters
and does not use the random sample points.

The sampled gradients $\nabla_\theta \hat J$ are computed with
$100$ points drawn independently and uniformly at each iteration.
We increase the sample count from the ten points used in
Figure~\ref{fig:elu-results} to reduce sampling fluctuations:
at fixed $\theta$, this reduces the sampling covariance by a
factor of ten while leaving the mean gradient $G$ unchanged.
The two sampled variants use the same sample points; the corrected
variant adds $K$, evaluated separately as described above.

All four variants use the ELU trial family and Adam optimizer of
Figure~\ref{fig:elu-results}, with
$(\beta_1,\beta_2,\varepsilon)=(0.9,0.999,10^{-8})$.
For each of six seeds ($0$--$5$), all variants
start from the same network initialization; Figure~\ref{fig:elu-results}
shows seed $3$. Each run lasts $20{,}000$ iterations, with cosine
decay of the learning rate from $10^{-3}$ to $10^{-5}$:
\[
\eta_t=10^{-5}+\frac{10^{-3}-10^{-5}}{2}
\left(1+\cos\frac{\pi t}{19{,}999}\right),
\qquad t=0,\ldots,19{,}999.
\]

To compute $G_Q$, we partition $[-6,6]$ at the computed interface
locations and subdivide each cell into intervals of length at
most $0.05$. We apply ten-point Gauss--Legendre quadrature on each
subinterval. The same quadrature procedure is used to evaluate
the loss for all four variants.
Figure~\ref{fig:ablation-2x2} shows the resulting loss histories.

\begin{figure}[!htbp]
\centering
\includegraphics[width=0.9\textwidth]{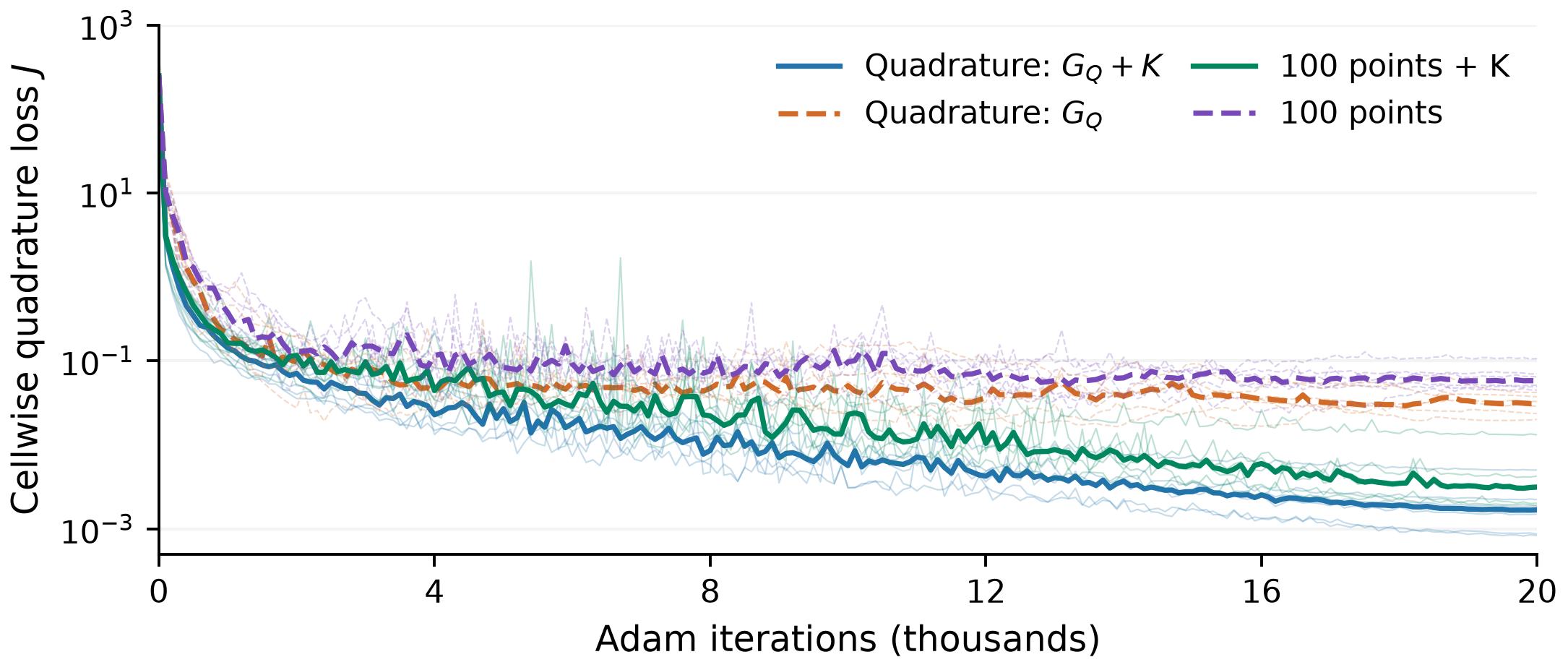}
\caption{{Controlled comparison for the Poisson experiment of
Figure~\ref{fig:elu-results}. Loss histories for the four
gradients handed to the optimizer (thick lines: medians over six seeds;
thin lines: individual seeds), with the same initializations,
optimizer, and cosine learning-rate schedule. The sampled variants
use $100$ points per iteration. Here $G_Q$ denotes the cellwise
quadrature approximation of $G$, and $K$ is the interface correction.}}
\label{fig:ablation-2x2}
\end{figure}

The terminal losses are listed in Table~\ref{tab:poisson-terminal-losses}.
Adding the interface term lowers the loss in every seed, for both
quadrature and sampled gradients. The reduction factors range from
$4.0$ to $49$ with quadrature and from $3.9$ to $34$ with $100$
samples. Replacing sampling by quadrature also lowers the loss
in every seed, with or without the interface correction. Thus the effect of the missing interface term remains
visible after reducing sampling fluctuations.

A separate comparison using ten samples per iteration and the full
$100{,}000$-iteration schedule of Figure~\ref{fig:elu-results}
gives terminal losses lower by factors of $12$--$66$ with the
interface correction than without it, across the six seeds.

In the quadrature runs without the interface term, $|G_Q|$ lies
between $0.035$ and $0.12$, while $|G_Q+K|$ lies between $3.6$ and
$8.2$. The mean sampled gradient is therefore small compared with
the evaluated continuous gradient. In the $100$-point sampled runs
without the interface correction, $|G_Q|$ lies between $1.9$ and
$5.9$, while the omitted term has magnitude $|K|$ between $4.8$
and $13$. The root-mean-square fluctuation of the sampled gradient
about its mean at the terminal iterate,
\[
\sigma_n:=\Bigl(\frac1n\Bigl(|\Omega|\int_\Omega r^2|\nabla_\theta r|^2\,dx-|G|^2\Bigr)\Bigr)^{1/2},
\qquad n=100,
\]
evaluated by the same cellwise quadrature, lies between $8.1$ and
$15.5$, or $0.71$--$1.60$ times $|G_Q+K|$. Sampling fluctuations
thus remain comparable to the continuous gradient despite the
larger sample count. The paired comparisons show that including
$K$ improves the attained loss even with reduced sampling noise;
quadrature further improves it under the stated schedule.

\section{Conclusion}\label{sec:conclusion}

The passage from residual optimization to a solution of a differential
equation involves several distinct implications. For a differentiable
residual map, stationarity requires the residual to be orthogonal to the
image of admissible variations under its linearization. Boundary
conditions and parameterization restrict those variations. A separate
question is whether the mean sampled gradient agrees with the gradient
of the continuous residual functional. For minimizing sequences, one
must determine whether convergent subsequences exist and whether their
limits retain the required regularity and attain the infimum. Even a
function with zero residual must satisfy the
admissibility conditions required of the intended solution.

In the transonic model, every minimizing sequence converges uniformly
to a function outside $H^1$ that is not a weak solution. The spatial
derivatives along the sequence have diverging $L^2$ norms. If bounded
parameters give an $L^2$ bound on these derivatives, approaching the
infimum requires parameter norms to diverge. Moving residual jumps give a
different obstruction: the continuous gradient can contain contributions
absent from the sampled gradient, independently of the sample size.
Under the stated assumptions, continuity of the residual removes this
bias. The favorable results for linear equations, the viscous model,
and uniformly convex Monge--Amp\`ere critical points show how additional
conditions justify the corresponding steps from optimization to a
solution.

For a particular method, a small loss gradient must be interpreted
through the available admissible variations and their conditioning.
Accurate sampled losses do not alone establish that the sampled
gradient agrees with the continuous gradient or control the continuous
residual along a training sequence. An error estimate further requires
stability of the residual formulation and conditions selecting the
intended solution. Increasing the number of training points or the
approximation capacity does not supply these properties by itself.
For equations selected by viscosity admissibility, minimizing the
residual of a consistent monotone discretization is one way to
incorporate that information, provided the discretization and
optimization errors satisfy the stability and refinement requirements
of the convergence analysis~\cite{esteve2025finite,bokanowski2026monotone,barles1991convergence}.

\section*{Acknowledgments}
Carlos Esteve-Yagüe is supported by Agencia Estatal de Investigación under the Ram\'on y Cajal 2022 grant RYC2022-035966-I (Spain).

Richard Tsai is supported partially by National Science Foundation grant DMS-2513857.

Stanley Osher is supported partially by National Science Foundation grants DMS-2208272 and DMS-1554564, and by
DOD-ARO W911NF241015,
DOE DE-SC0026262,
DOD-DARPA HR00112590074.

\bibliographystyle{abbrv} 
\bibliography{references}

\end{document}